\documentclass[11pt,a4paper]{article}

\usepackage[T1]{fontenc}
\usepackage[utf8]{inputenc}
\usepackage{lmodern}
\usepackage[margin=25mm]{geometry}
\usepackage{amsmath,amssymb,amsthm}
\usepackage{array}
\usepackage{microtype}
\usepackage[numbers,sort&compress]{natbib}
\usepackage[hidelinks]{hyperref}

\newcommand{\WLAncillaryMaterials}{}

\numberwithin{equation}{section}

\theoremstyle{plain}
\newtheorem{theorem}{Theorem}[section]
\newtheorem{lemma}[theorem]{Lemma}
\newtheorem{proposition}[theorem]{Proposition}
\newtheorem{corollary}[theorem]{Corollary}
\theoremstyle{definition}
\newtheorem{definition}[theorem]{Definition}
\theoremstyle{remark}
\newtheorem{remark}[theorem]{Remark}

\newcommand{\rw}{\operatorname{rw}}
\newcommand{\lrw}{\operatorname{lrw}}
\newcommand{\ebw}{\operatorname{ebw}}

\newcommand{\BP}{\mathrm{BP}}
\newcommand{\Clog}{\mathsf{C}}

\title{Decomposition Profiles and Weisfeiler-Leman Dimension for Graphs of Bounded Rank Width}
\author{Antonios Kalampakas\\[0.4em]
  \small College of Engineering and Technology\\
  \small American University of the Middle East\\
  \small Egaila 54200, Kuwait\\[0.4em]
  \small \href{mailto:antonios.kalampakas@aum.edu.kw}{\texttt{antonios.kalampakas@aum.edu.kw}}\\
  \small \href{https://orcid.org/0000-0001-8821-6102}{ORCID: 0000-0001-8821-6102}}
\date{}

\hypersetup{
  pdftitle={Decomposition Profiles and Weisfeiler-Leman Dimension for Graphs of Bounded Rank Width},
  pdfauthor={Antonios Kalampakas},
  pdfsubject={Graph theory and finite model theory},
  pdfkeywords={Weisfeiler-Leman algorithm, graph isomorphism, decomposition profiles, rank-width, linear rank-width, finite model theory, counting logic, bijective pebble game, Cai-Furer-Immerman graphs}}

\begin{document}

\maketitle
\begin{abstract}
The Weisfeiler-Leman (WL) algorithm tests graph isomorphism by assigning colours to $d$-tuples of vertices and iteratively refining the colouring according to neighbourhood configurations until it stabilizes. Increasing $d$ allows the test to detect finer structural differences between graphs, but the computational cost grows exponentially with $d$. The central question is which dimension allows us to distinguish nonisomorphic graphs. We relate a sufficient such dimension to the structure of a graph $G$ through a rank decomposition over $\mathbb F_2$. The two parameters governing the bound are the maximum cut rank $w$ and the fork load $\ell$, which is the largest sum of the parent and two child cut ranks at a fork in the decomposition tree. We prove that WL in dimension $\max\{\ell+1,2w+2\}$ distinguishes every finite nonempty vertex-coloured graph $G$ with such a decomposition from every graph not isomorphic to $G$. Consequently, for $k\geq1$, graphs of rank width at most $k$ can be distinguished from every nonisomorphic graph in dimension $3k+1$, graphs of linear rank width at most $k$ in dimension $2k+2$, and edgeless graphs in dimension one. Uncoloured Cai-F{\"u}rer-Immerman (CFI) constructions give linear lower bounds for both width parameters, so the upper bounds are tight up to constant factors.
\end{abstract}

\noindent\textit{Keywords:} Weisfeiler-Leman algorithm, graph isomorphism,
decomposition profiles, rank-width, linear rank-width, finite model theory,
counting logic, bijective pebble game, Cai-F{\"u}rer-Immerman graphs.

\section{Introduction}
\label{sec:introduction}

The Weisfeiler-Leman procedure is a central bridge between graph
isomorphism, finite model theory, and structural graph theory. Its dimension
measures how many vertices must be considered simultaneously to distinguish a
graph from every nonisomorphic graph. The $d$-dimensional procedure has
equivalent formulations through the $(d+1)$-variable fragment of counting
logic and the bijective $(d+1)$-pebble game. Bounds on the required dimension
therefore yield both combinatorial identification results and bounded-variable
definability statements.

Rank width is particularly suited to this question. A rank decomposition
controls every cut by the rank over $\mathbb F_2$ of its bipartite adjacency
matrix, although the corresponding interface need not be a small separator in
the usual graph-theoretic sense. Oum and Seymour introduced rank width within a
general decomposition framework for dense graphs \cite{OumSeymour2006}.
Grohe and Schweitzer gave the first polynomial-time isomorphism test for every
class of bounded rank width \cite{GroheSchweitzer2015}. Grohe and Neuen later
developed a source-guided method based on split pairs, colour refinement, and
flips \cite{GroheNeuen2023}. Their published result shows that
$(3k+4)$-dimensional Weisfeiler-Leman identifies every graph of rank width at
most $k$. Corollary~3.12 in the full version of the SWAT paper by Levet,
Rombach, and Sieger improves this unrestricted bound to $3k+3$
\cite{LevetRombachSieger2024Full}.

Our main result responds to the ranks displayed by the particular source
decomposition. For a rooted rank decomposition $\mathcal D$, let
$w(\mathcal D)$ be its maximum cut rank and let $\ell(\mathcal D)$ be the
maximum, over its forks, of the sum of the parent and two child cut ranks.
Writing $\dim_{\mathrm{WL}}(G)$ for the least dimension that identifies $G$,
we prove
\[
  \dim_{\mathrm{WL}}(G)
  \leq
  \max\{\ell(\mathcal D)+1,2w(\mathcal D)+2\}.
\]
Since $\ell(\mathcal D)\leq3w(\mathcal D)$, every graph of rank width at
most $k$ is identified by $(3k+1)$-dimensional Weisfeiler-Leman when
$k\geq1$. A linear layout gives a rooted-comb decomposition with a singleton
child side at each fork, which yields the stronger $(2k+2)$-dimensional bound
for linear rank width at most $k$. Rank-width-zero graphs are edgeless and are
identified by one-dimensional Weisfeiler-Leman. Thus the general rank-width
estimate improves the strongest previous unrestricted bound by two
dimensions. Equivalently, individual graphs in the two classes are definable
in counting logic with $3k+2$ and $2k+3$ variables, respectively. Only the
source graph is equipped with a rank decomposition. The target graph remains
arbitrary throughout the proof.

The proof refines the nice-triple geometry underlying the induction of Grohe
and Neuen. The same three-interface containment appears in the pebble count of
Levet, Rombach, and Sieger
\cite[Section~3.2]{LevetRombachSieger2024Full}. At an internal node, let $S$
be the parent split pair and let $F$ and $T$ be the split pairs for its two
children. These pairs can be chosen so that their distinct entries lie in one
common carrier $P$. If the three cut ranks are $r,r_1,r_2$, then
\[
  |P|\leq r+r_1+r_2.
\]
Entries absent from a child interface lie outside the corresponding child
side. This remoteness property turns the local rank sum into the global
fork-load bound.

The main obstacle is synchronizing component information across the two
graphs. Stable colourings formed from different individualized tuples cannot
be compared by refinement alone. We use canonical common-round colour codes
and attach to each relevant stable colour its component coordinate in the
appropriate flipped graph. The resulting component palettes have the same
meaning on both sides of the pebble game. Two palettes are preannounced in
sequence. Each either exposes a response that violates the prescribed data or
forces a nonisomorphic coloured component for the next child transition. A
canonical forgetful map recovers information defined with a finer tuple from
the palette on the selected remote component. This avoids simultaneous
markers for both children.

At a fork $t$, the carrier, current marker, and response use at most
$\ell_{\mathcal D}(t)+2$ pebbles. A marked split-pair palette or component
formula uses at most $2w(\mathcal D)+3$ variables. The strict-child transition
therefore iterates with
\[
  \max\{\ell(\mathcal D)+2,2w(\mathcal D)+3\}
\]
pebbles. Root, leaf, and finitary closure arguments then give the stated
identification theorem.

The dependence on both width parameters is optimal up to constant factors.
Let $D_{\mathrm{rw}}(r)$ and $D_{\mathrm{lrw}}(r)$ be the maximum
Weisfeiler-Leman dimensions on graphs of rank width and linear rank width at
most $r$. We apply the original Cai-F{\"u}rer-Immerman (CFI) construction
\cite{CaiFurerImmerman1992} in the uncoloured form developed by Neuen
\cite{Neuen2024Homomorphism}. A block-layout analysis on the split grids of
Kiefer and Neuen \cite[Section~6.2]{KieferNeuen2022} gives
\[
  \max\left\{2,\left\lfloor\frac r2\right\rfloor-1\right\}
  \leq D_{\mathrm{lrw}}(r)\leq2r+2
\]
for every $r\geq2$. Factoring cuts through the endpoints of crossing base
edges yields
\[
  \max\left\{3,\left\lfloor\frac r2\right\rfloor\right\}
  \leq D_{\mathrm{rw}}(r)\leq3r+1
\]
for every $r\geq4$. The lower construction uses the fact that the $n$ by $n$
square grid has tree width $n$ \cite[Section~13.2]{Bodlaender1998}. The rook
and Shrikhande graphs provide the constant-three calibration, including
$D_{\mathrm{lrw}}(r)\geq3$ for $r\geq5$. These estimates show that both
extremal dimensions grow linearly.

The standard implementation of Weisfeiler-Leman consequently tests
isomorphism on the rank-width-$k$ class in time $O(n^{3k+2}\log n)$, improving
the previous exponent by two. The broader lower-bound context includes
oddomorphisms for homomorphism indistinguishability \cite{Roberson2026} and
compressed CFI graphs designed for strong
Weisfeiler-Leman lower bounds \cite{GroheLichterNeuenSchweitzer2025}. Our
analysis instead tracks rank width and linear rank width closely enough to
place the indistinguishable pairs inside the corresponding width-bounded
classes.

The decomposition-sensitive estimate also yields structural information after
optimization. Define
\[
  \ell^\star(G)=\min_{\mathcal D}\ell(\mathcal D),
  \qquad
  d^\star(G)=\min_{\mathcal D}
  \max\{\ell(\mathcal D)+1,2w(\mathcal D)+2\}.
\]
We prove
\[
  2\rw(G)\leq\ell^\star(G)\leq3\rw(G),
\]
and $d^\star(G)$ lies between $\ell^\star(G)+1$ and
$\ell^\star(G)+2$. We call $(\ell^\star(G),d^\star(G))$ the optimized
profile pair. Every connected graph $G$ on at least three vertices has, for
each $q\geq1$, a connected induced supergraph $G_q$ with the same rank width
and optimized profile pair, while
\[
  \lrw(G_q)=\max\{\lrw(G),q\}.
\]
Using square grids as base graphs gives, for every $r\geq1$ and
$q\geq r+1$, a connected graph with rank width $r$ and linear rank width $q$
such that
\[
  \ell^\star(G_q)\leq\min\{3r,2r+3\},
  \qquad
  2r+2\leq d^\star(G_q)\leq\min\{3r+1,2r+4\}.
\]
Thus linear rank width can grow without changing the optimized profile.

Both optimized quantities obey exact maximum identities under disjoint union
and nontrivial split composition. The corresponding identity for rank width
is classical \cite[Theorem~4.3]{HlinenyEtAl2008}. Iteration over an atomic
split decomposition gives a maximum formula for every nonempty graph. At rank
width exactly two, the optimized profile pair is always $(5,6)$ or $(6,7)$,
and both pairs occur at every exact linear rank width at least three. Hence
global linear rank width does not determine the optimized profile pair, and
decomposition optimization alone cannot lower the endpoint seven for graphs
of rank width exactly two.

These structural results complement recent work on recognizability and
counting monadic second-order logic (CMSO) definability at rank width at most
two, and on sequential cable constructions for linear rank width
\cite{Kalampakas2026CMSO,Kalampakas2026Cable}. Those works do not study the
optimized profile pair introduced here.

The theorem concerns identification after unrestricted Weisfeiler-Leman
stabilization. It does not bound the number of refinement rounds or determine
the optimal dimension. Levet, Rombach, and Sieger obtain a different tradeoff
by using $(6k+3)$-dimensional Weisfeiler-Leman for $O(\log n)$ rounds in a
parallel canonization framework \cite{LevetRombachSieger2024}. Our argument
does not produce a canonical target decomposition, a bounded-depth
description, or fixed-point definable canonization. In particular, for graphs
of rank width at most two it gives a seven-dimensional upper bound without
determining the exact dimension of the class.

Section~\ref{sec:preliminaries} fixes the game, logic, rank-width, split-pair,
and flip conventions. Section~\ref{sec:carrier-synchronization} proves the
carrier bound and the remote synchronization tools.
Section~\ref{sec:key-enforcement} interprets every auxiliary key in the
ordinary graph vocabulary. Section~\ref{sec:transition} establishes the
two-palette strict-child transition. Section~\ref{sec:induction} closes the
source-guided induction and proves the main theorem.
Section~\ref{sec:optimized-profiles} develops optimized fork load, separation,
split composition, and the rank-width-two classification.
Section~\ref{sec:quantitative-consequences} gives the extremal dimension
ranges, finite calibration, and algorithmic consequences.
Section~\ref{sec:conclusion} records the remaining questions.

\section{Preliminaries}
\label{sec:preliminaries}

We work with finite simple undirected graphs over one fixed finite vocabulary.
Besides the symmetric irreflexive edge relation, the vocabulary may contain
finitely many unary vertex colours.  Isomorphisms and partial isomorphisms
must preserve all of these relations.  The uncoloured case is obtained by
using no unary predicates.
Whenever a rank decomposition or an optimized profile pair is discussed, the
graph is understood to have a nonempty vertex set. If the empty graph is
admitted in the ambient graph convention, it is the unique graph of its order
and is identified trivially.

For vertex-coloured graphs $G$ and $H$ of the same order and an integer
$m\geq1$, the bijective $m$-pebble game is played with $m$ indexed pairs of
pebbles. A position is a pair of tuples
$((g_i)_{i\in I},(h_i)_{i\in I})$ for some $I\subseteq[m]$, where
$g_i\in V(G)$ and $h_i\in V(H)$. Entries may repeat. The position is a
partial isomorphism when equality, adjacency, and every unary input colour
agree at all corresponding indices. In one round Spoiler either removes one
occupied pair, in which case the round ends, or requests a new pair when
fewer than $m$ pairs are present. In the latter case Duplicator announces an
arbitrary bijection from $V(G)$ to $V(H)$. Spoiler then chooses a vertex in
the source graph and places the new source pebble there. The matching target
pebble is placed on its image under the announced bijection. Spoiler wins as
soon as the new position is not a partial isomorphism.
Duplicator wins an infinite play if every position remains a partial
isomorphism.  We write $\BP_m(G,H)$ for this game and invoke it only for
equal-order graphs. Graphs of different orders are already separated by a
one-variable counting sentence.

Counting logic is first-order logic augmented with quantifiers of the form
$\exists^{\geq t}x$.  Its fragment using at most $m$ variable names is
denoted by $\Clog^m$.  Rebinding a name in a nested subformula does not bind
an occurrence of that name in a Boolean sibling.  This lexical convention is
used repeatedly in the variable counts below.  We use the standard
correspondence between finite-variable counting logic, bijective pebble
games, and Weisfeiler-Leman refinement
\cite{ImmermanLander1990,CaiFurerImmerman1992,Hella1996}.

\subsection{Rank decompositions, split pairs, and flips}
\label{subsec:source-preliminaries}

Rank width is always taken on the underlying uncoloured graph.  We follow the conventions of
\cite[Sections~2.2, 3, and~4]{GroheNeuen2023}.

For $X\subseteq V(G)$, write $\overline X=V(G)\setminus X$ and let
\[
  M_G(X,\overline X)\in
  \mathbb F_2^{X\times\overline X}
\]
be the bipartite adjacency matrix across the cut.  The cut rank is
\[
  \rho_G(X)=\operatorname{rk}_{\mathbb F_2}M_G(X,\overline X).
\]
A rank decomposition of a nonempty graph is a pair $(T,\gamma)$ in which
$T$ is a finite full rooted binary tree and $\gamma$ assigns a set of
vertices to each node.  At the root this set is $V(G)$.  At every internal
node it is the disjoint union of the sets assigned to the two children.  The
sets at the leaves are singletons.
The width of $(T,\gamma)$ is
\[
  \max_{t\in V(T)}\rho_G(\gamma(t)),
\]
and $\rw(G)$ is the minimum of this quantity over all rank
decompositions.  This is the rooted presentation used in
\cite[Section~2.2]{GroheNeuen2023}.

For a fixed rooted rank decomposition $\mathcal D=(T,\gamma)$, put
\[
  X_t=\gamma(t),
  \qquad
  r_{\mathcal D}(t)=\rho_G(X_t),
  \qquad
  w(\mathcal D)=\max_{t\in V(T)}r_{\mathcal D}(t).
\]
If $t$ has children $t_1,t_2$, define its fork load by
\[
  \ell_{\mathcal D}(t)
  =r_{\mathcal D}(t)+r_{\mathcal D}(t_1)+r_{\mathcal D}(t_2).
\]
The fork load of the decomposition is
\begin{equation}
  \ell(\mathcal D)
  =\max_{t\text{ internal}}\ell_{\mathcal D}(t).
  \label{eq:decomposition-fork-load}
\end{equation}
When the tree has no internal node, we set \(\ell(\mathcal D)=0\).
Thus $\ell(\mathcal D)\leq3w(\mathcal D)$. If the root is internal, its
first summand is zero and its two child cuts have equal rank.
Fomin and Korhonen define a related sum-width for the tripartitions displayed
by a $W$-improvement of a branch decomposition
\cite[Definition~4.3]{FominKorhonen2024}. Here fork load is the maximum,
over the internal nodes of one rooted rank decomposition, of the sum of the
three incident cut ranks.
Set
\begin{equation}
  m(\mathcal D)
  =\max\{\ell(\mathcal D)+2,\,2w(\mathcal D)+3\}.
  \label{eq:decomposition-pebble-budget}
\end{equation}

A linear layout of $G$ is an ordering $(v_1,\ldots,v_n)$ of its vertices.
Its width is
\[
  \max_{1\leq i<n}\rho_G(\{v_1,\ldots,v_i\}).
\]
The maximum is defined to be zero when \(n\leq1\).
The linear rank width $\lrw(G)$ is the minimum width of a linear layout.
Equivalently, it is obtained by restricting rank decompositions to
caterpillar decomposition trees \cite{JeongKwonOum2014}.

For $Y\subseteq V(G)$ and $v\in Y$, let
\[
  \operatorname{vec}_Y(v)
  =\bigl(\mathbf 1_{vw\in E(G)}\bigr)_{w\in\overline Y}
  \in\mathbb F_2^{\overline Y}.
\]
A split pair for $X$ is a pair $(A,B)$ such that $A\subseteq X$ and
$B\subseteq\overline X$, the vectors
$\{\operatorname{vec}_X(a):a\in A\}$ form a basis of the row space of
$M_G(X,\overline X)$, and
$\{\operatorname{vec}_{\overline X}(b):b\in B\}$ form a basis of the row
space of $M_G(\overline X,X)$.  Consequently,
\[
  |A|=|B|=\rho_G(X).
\]
An ordered split pair is an ordering $(\bar a,\bar b)$ of these two sets.  We
also regard the concatenation $(\bar a,\bar b)$ as an ordered anchor
occurrence and denote it by $\mathbf A$.  This agrees with
\cite[Definition~3.2]{GroheNeuen2023}.  The empty pair is the split pair at
the root.

Let $\mathcal U$ be the ordered list of unary relation symbols in the fixed
vocabulary. For $v\in V(G)$, define its complete unary atomic type by
\[
  \chi_G(v)=\bigl(\mathbf 1_{v\in P^G}\bigr)_{P\in\mathcal U}.
\]
For an ordered anchor occurrence
$\mathbf q=(q_1,\ldots,q_r)$, let
\[
  \chi_G^{\mathbf q}(v)
  =\bigl(\chi_G(v),(\mathbf 1_{v=q_i})_{i\in[r]}\bigr)
\]
and define the round-zero canonical colour by
\[
  c^{G,\mathbf q}_0(v)=\chi_G^{\mathbf q}(v).
\]
Inductively, put
\[
  c^{G,\mathbf q}_{i+1}(v)
  =\left(c^{G,\mathbf q}_i(v),
    \bigl\{\!\!\bigl\{c^{G,\mathbf q}_i(w):w\in N_G(v)
    \bigr\}\!\!\bigr\}\right).
\]
These are common recursive codes rather than graph-local names.  We write
\[
  \lambda^G_{\mathbf q}(v)
  =\operatorname{Stab}_G(\chi_G^{\mathbf q})(v)
\]
for the resulting canonical stable colour.  Formally, it may be represented
by the common-round code sequence
$(c^{G,\mathbf q}_i(v))_{i\geq 0}$.  On a finite graph it induces the usual
stable colour-refinement partition.  Repeated anchor occurrences cause no
ambiguity because their complete equality vector is retained.  This is a
canonical naming of the individualized stable colouring used in
\cite[Section~3, before Corollary~3.4]{GroheNeuen2023}.

If $\mathbf q$ is obtained from $\mathbf u$ by a fixed map of source
occurrences, there is a graph-independent forgetful map
$\pi_{\mathbf u,\mathbf q}$ satisfying
\[
  \lambda^G_{\mathbf q}
  =\pi_{\mathbf u,\mathbf q}\circ\lambda^G_{\mathbf u}.
\]
At round zero the map retains, repeats, and reorders the prescribed equality
coordinates.  At every later round it aggregates neighbour counts over the
fibres of the preceding map.  We use only this fine-to-coarse direction.

Let $\alpha:V(G)\to\mathcal C$ be a colouring.  A flip function for
$(G,\alpha)$ is a symmetric map
\[
  f:\mathcal C^2\longrightarrow\{0,1\}.
\]
The flipped graph $G^{\alpha,f}$ has the same vertices and colours as $G$,
and for distinct $u,v$ its edge relation is given by
\[
  \mathbf 1_{uv\in E(G^{\alpha,f})}
  =\mathbf 1_{uv\in E(G)}\oplus
    f(\alpha(u),\alpha(v)).
\]
When the colouring is clear, we write $G^f$.  This is
\cite[Definition~3.5]{GroheNeuen2023}.

Fix the source graph $G$.  If $\mathbf A=(\bar a,\bar b)$ is an ordered split
pair for $X$, choose once and for all a flip function supplied by
\cite[Lemma~3.6]{GroheNeuen2023} for the colouring
$\lambda^G_{\mathbf A}$.  Denote it by
\[
  f_{X,\mathbf A}.
\]
Extend it by zero to the universal alphabet of canonical stable types.  The
same totalized map is then evaluated on every target graph.  This convention
does not claim that the target components respect a target cut.  It only
ensures that the source flip and target flip are the same type-defined
operation.

If $\mathbf A'$ is the target occurrence paired with $\mathbf A$, then
$H^{f_{X,\mathbf A}}$ means that the totalized map is evaluated on the
canonical target colouring $\lambda^H_{\mathbf A'}$.  In particular, a
canonical type realised only on the target is governed by the zero extension.

If a canonical refinement $\beta$ comes with a fixed projection $\pi$ from
its universal stable-type alphabet to the universal stable-type alphabet of
$\lambda_{\mathbf A}$, the lift of $f_{X,\mathbf A}$ to $\beta$ is
\[
  f_{X,\mathbf A}^{\beta}(c,d)
  =f_{X,\mathbf A}(\pi(c),\pi(d)).
\]
The lift toggles exactly the same unordered vertex pairs as the original
flip.  We therefore suppress the superscript when the projection is clear.

Let
\[
  \mathcal C_{\alpha,f}(G)
\]
denote the family of vertex sets of the connected components of
$G^{\alpha,f}$.  For $y\in V(G)$, let
\[
  [\mathbf A,y]^G_{f_{X,\mathbf A}}
\]
be the unique member of
$\mathcal C_{\lambda^G_{\mathbf A},f_{X,\mathbf A}}(G)$ containing
$y$.  We omit $G$, $X$, or the subscript only when they are fixed by the
context.  The analogous target component containing $y'$ is written
\[
  [\mathbf A',y']^H_{f_{X,\mathbf A}}.
\]

Suppose that $X=X_1\mathbin{\dot\cup}X_2$, that $(A,B)$ is a split pair for
$X$, and that $(A_i,B_i)$ is a split pair for $X_i$.  The two child pairs are
nice with respect to $(A,B)$ when
\[
  A\cap X_i\subseteq A_i
  \quad\text{and}\quad
  B_{3-i}\cap X_i\subseteq A_i
  \qquad (i\in\{1,2\}).
\]
Grohe and Neuen prove that they may be chosen so that, in addition,
\[
  B_i\cap\overline X\subseteq B
  \qquad (i\in\{1,2\}).
\]
See \cite[Definition~4.1 and Lemma~4.2]{GroheNeuen2023}.  These inclusions
are converted into the shared-carrier and remoteness bounds in
Section~\ref{sec:carrier}.

We record the external results used below in the precise forms needed here.

\begin{proposition}[Split-pair flip]
\label{prop:gn-split-pair-flip}
Let $\mathbf A$ be an ordered split pair for $X$.  Every component in
\[
  \mathcal C_{\lambda^G_{\mathbf A},f_{X,\mathbf A}}(G)
\]
is contained in $X$ or in $\overline X$.
\end{proposition}

\noindent
This is \cite[Lemma~3.6]{GroheNeuen2023}.

\begin{proposition}[Isomorphism invariance under a common flip]
\label{prop:gn-flip-isomorphism}
Let $(G,\alpha)$ and $(H,\alpha')$ use a common colour alphabet and let $f$
be a flip function on that alphabet.  For every bijection
$\varphi:V(G)\to V(H)$,
\[
  \varphi:(G,\alpha)\cong(H,\alpha')
  \quad\Longleftrightarrow\quad
  \varphi:(G^{\alpha,f},\alpha)
  \cong(H^{\alpha',f},\alpha').
\]
\end{proposition}

\noindent
This is \cite[Lemma~3.9]{GroheNeuen2023}.  We will also use its
contrapositive for induced coloured subgraphs.

\begin{proposition}[Colour-refinement invariance under a flip]
\label{prop:gn-flip-colour-refinement}
For every coloured graph $(G,\alpha)$ and every flip function $f$ for
$(G,\alpha)$, the two stable colourings
\[
  \operatorname{Stab}_G(\alpha)
  \quad\text{and}\quad
  \operatorname{Stab}_{G^{\alpha,f}}(\alpha)
\]
induce the same partition of $V(G)$.
\end{proposition}

\noindent
This is \cite[Corollary~3.11]{GroheNeuen2023}.  The assertion concerns
equality of the induced partitions.  Our canonical common-round codes supply
the cross-graph type names used later.

\begin{proposition}[Nice child split pairs]
\label{prop:gn-nice-split-pairs}
Let $X=X_1\mathbin{\dot\cup}X_2$ and let $(A,B)$ be a split pair for $X$.
There are child split pairs $(A_i,B_i)$ for $X_i$ that are nice with respect
to $(A,B)$ and satisfy
\[
  B_i\cap\overline X\subseteq B
  \qquad (i\in\{1,2\}).
\]
\end{proposition}

\noindent
This is \cite[Lemma~4.2]{GroheNeuen2023}.

\begin{proposition}[Selection from component partitions]
\label{prop:gn-component-selection}
Let $G_0$ and $H_0$ be nonisomorphic coloured graphs of the same order.  Let
$\mathcal P$ and $\mathcal Q$ be partitions of their vertex sets such that
every connected component is contained in one block.  For every bijection
$\sigma:V(G_0)\to V(H_0)$, there is a vertex $v\in V(G_0)$ for which
\[
  G_0[P_v]\not\cong H_0[Q_{\sigma(v)}],
\]
where $P_v$ and $Q_{\sigma(v)}$ are the unique blocks containing $v$ and
$\sigma(v)$.
\end{proposition}

\begin{proof}
This is the coloured form of
\cite[Observation~4.3]{GroheNeuen2023}.  Suppose the conclusion fails.  Form
the bipartite incidence graph on $\mathcal P$ and $\mathcal Q$, joining two
blocks when $\sigma$ maps at least one vertex between them.  Every incidence
edge joins blocks of the same coloured isomorphism type.  The bijection
$\sigma$ therefore maps the union of the source blocks of each fixed type
onto the union of the target blocks of that type.  Since all blocks of one
type have the same order, the two partitions contain equally many blocks of
that type.  Pair the blocks type by type and combine coloured block
isomorphisms.  No edge joins distinct blocks of a component partition, so
the combined map is an isomorphism from $G_0$ to $H_0$, a contradiction.
\end{proof}

Write $G\simeq_d H$ when $d$-dimensional Weisfeiler-Leman does not
distinguish $G$ and $H$. We say that $d$-dimensional Weisfeiler-Leman
identifies $G$ if $G\not\simeq_d H$ for every graph $H$ over the same
vocabulary that is not isomorphic to $G$. The Weisfeiler-Leman dimension of
$G$ is
\[
  \dim_{\mathrm{WL}}(G)
  =\min\{d\geq1 : d\text{-dimensional Weisfeiler-Leman identifies }G\}.
\]
Write $G\equiv_{\Clog^r}H$ when the two graphs
satisfy the same sentences of $r$-variable first-order logic with counting.

\begin{proposition}[Weisfeiler-Leman, pebble games, and counting logic]
\label{prop:gn-wl-game-logic}
For vertex-coloured graphs $G$ and $H$ of the same order and every $d\geq1$,
\[
  G\simeq_d H
  \quad\Longleftrightarrow\quad
  \text{Duplicator wins }\BP_{d+1}(G,H)
  \quad\Longleftrightarrow\quad
  G\equiv_{\Clog^{d+1}}H.
\]
\end{proposition}

\noindent
The two equivalences are
\cite[Theorems~2.3 and~2.4]{GroheNeuen2023}.  Consequently, if Spoiler wins
$\BP_m(G,H)$ against every same-order nonisomorphic target
$H$ for $m\geq2$, then $(m-1)$-dimensional Weisfeiler-Leman identifies $G$.  Pairwise
$\Clog^m$ separators combine into one identifying sentence by the finite
conjunction argument in \cite[Corollary~2.5]{GroheNeuen2023}.  Graphs of a
different order are separated by a one-variable counting sentence.

For comparison, \cite[Theorem~4.4]{GroheNeuen2023} proves that
$(3k+4)$-dimensional Weisfeiler-Leman identifies graphs of rank width at most
$k$.  Its proof uses the same source-guided descent and a
$\BP_{3k+5}$ strategy. Corollary~3.12 of the full version of
Levet, Rombach, and Sieger sharpens the unrestricted dimension bound to
$3k+3$ \cite{LevetRombachSieger2024Full}. We use the theorem of Grohe and Neuen as
the source of the induction architecture.

\section{Shared carriers and remote synchronization}
\label{sec:carrier-synchronization}

\subsection{The shared carrier and remoteness}
\label{sec:carrier}

Fix a source graph $G$, a graph $H$ of the same order, and a rooted rank
decomposition $\mathcal D=(\mathcal T,\gamma)$ of $G$. Consider an internal
decomposition node $t$ with shore $X$ and children $t_1,t_2$ with shores
$X_1,X_2$, where
\[
X=X_1\mathbin{\dot\cup}X_2.
\]
Let
\begin{equation}
\mathbf S=(\bar a,\bar b),\qquad
\mathbf F=(\bar a_1,\bar b_1),\qquad
\mathbf T=(\bar a_2,\bar b_2)
\label{eq:parent-child-occurrences}
\end{equation}
be the ordered split-pair occurrences for the parent, first child, and second child. We choose them using the strengthened niceness clause of \cite[Lemma~4.2]{GroheNeuen2023}. Their numbers of entries are twice the corresponding cut ranks. When an occurrence tuple appears in a set operation, it denotes the set of its entries.

Let $\mathbf P$ be the ordered carrier obtained from
$\mathbf S\cup\mathbf F\cup\mathbf T$ by retaining one copy of every
distinct source vertex in a fixed source-determined order.

\begin{lemma}[Shared carrier]
\label{lem:shared-carrier}
The split-pair occurrences in \eqref{eq:parent-child-occurrences} can be chosen so that
\[
  |\mathbf P|\leq
  r_{\mathcal D}(t)+r_{\mathcal D}(t_1)+r_{\mathcal D}(t_2),
\]
where $t_1,t_2$ are the children of the current node $t$.
\end{lemma}

\begin{proof}
Write the parent split pair as $(A,B)$, where $A\subseteq X$ and $B\subseteq\overline X$. Write the child split pairs as $(A_1,B_1)$ and $(A_2,B_2)$, where $A_i\subseteq X_i$. The additional clause in \cite[Lemma~4.2]{GroheNeuen2023} gives
\begin{equation}
B_i\cap\overline X\subseteq B
\qquad\text{for }i\in\{1,2\}.
\label{eq:external-niceness}
\end{equation}
Together with niceness, this yields
\begin{equation}
A\subseteq A_1\cup A_2,
\qquad
B_1\subseteq A_2\cup B,
\qquad
B_2\subseteq A_1\cup B.
\label{eq:nice-triple-containment}
\end{equation}
The sets $A_1$, $A_2$, and $B$ lie in the pairwise disjoint regions $X_1$, $X_2$, and $\overline X$. Their sizes are
$r_{\mathcal D}(t_1)$, $r_{\mathcal D}(t_2)$, and
$r_{\mathcal D}(t)$, respectively. Every vertex occurring in
$\mathbf S\cup\mathbf F\cup\mathbf T$ belongs to
$A_1\cup A_2\cup B$, which proves the bound.
\end{proof}

\begin{corollary}[Profile carrier bound]
\label{cor:profile-carrier}
At every internal node $t$,
\[
  |\mathbf P|\leq\ell_{\mathcal D}(t)\leq\ell(\mathcal D).
\]
\end{corollary}

\begin{proof}
This is Lemma~\ref{lem:shared-carrier} and the definition of the fork load.
\end{proof}

\begin{corollary}[Remoteness]
\label{cor:carrier-remoteness}
The shared carrier satisfies
\begin{equation}
\mathbf P\setminus\mathbf F\subseteq\overline{X_1},
\qquad
\mathbf P\setminus\mathbf T\subseteq\overline{X_2}.
\label{eq:carrier-remoteness}
\end{equation}
\end{corollary}

\begin{proof}
Use the notation $A_1,A_2,B$ from the proof of
Lemma~\ref{lem:shared-carrier}. Every vertex of $\mathbf P$ lies in
$A_1\cup A_2\cup B$. Since $A_1$
occurs in $\mathbf F$, removing the vertices of $\mathbf F$ leaves only
vertices in $A_2\cup B$. The other inclusion is symmetric.
\end{proof}

Suppose the current position contains $\mathbf S$ and a marker $v$. Spoiler
acquires the distinct vertices in
$\mathbf P\setminus(\mathbf S\cup\{v\})$ one at a time. The target
occurrences are obtained from the replies through the fixed source occurrence
maps. After each reply, Spoiler checks the full enlarged position, including
all incidences between the newly acquired carrier vertex and the old marker.
A disagreement in equality, adjacency, or input colour is an atomic win.
Otherwise the resulting position contains at most $\ell(\mathcal D)+1$ live
roles, represented by $\mathbf P$ and $v$. This construction uses no decomposition
of $H$.

We use the semantic stable palettes $\lambda_{\mathbf A}$ defined in
Section~\ref{subsec:source-preliminaries}. A palette is the combined
colouring determined by one canonical colouring or by a finite tuple of
canonical colourings and Boolean coordinates. We call a palette stable when
this combined colouring is stable under colour refinement in the graph under
consideration. In particular, every stable palette is equitable. A factor
of a palette is a colouring obtained from it by one fixed map on colours.
The histogram of a palette records the multiplicity of each canonical
colour. For a colouring $\alpha$ and an ordered occurrence $\mathbf Y$, the notation
$\alpha^{\mathbf Y}$ means that the equality vector of $\mathbf Y$ is added to the initial
colours. We write $\alpha\equiv\beta$ when the two colourings induce the same
partition. Whenever an argument passes through this partition equivalence,
we relabel one colour set by the induced bijection between its classes before
comparing colour maps. This relabelling changes neither the partition nor any
component predicate.

All lookups are source-fixed maps on canonical stable types. Their finite
common-round representatives and counting-logic formulas are constructed in
Section~\ref{sec:key-enforcement}. The finite-to-complete passage is stated
formally in Lemma~\ref{lem:finite-complete-type-bridge}. Source-derived flips, their zero
extensions, and their lifts use the conventions of
Section~\ref{subsec:source-preliminaries}. A common lifted flip means one
source-fixed flip table that is zero-extended and pulled back on both graphs
through the same source-fixed map on their common canonical colour alphabet.
It therefore toggles corresponding colour pairs by the same rule. In
particular, later refinement does not change the flipped edge relation.

We make the component conventions explicit. For a graph $K$ and a fixed flip $f$, let
\[
\mathcal C_f(K)
=
\{D:D\text{ is a connected component of }K^f\}
\]
be the full-component partition of the flipped graph. For a flip
$f_{\mathbf A}$ defined through an anchor occurrence $\mathbf A$, we also
write
$\mathcal C_{\mathbf A,f_{\mathbf A}}(K)=\mathcal C_{f_{\mathbf A}}(K)$.
For $W\subseteq V(K)$, define the associated intersection partition by
\begin{equation}
\mathcal C_{\mathbf A,f_{\mathbf A}}(K)\mathbin{|}W
=
\{W\cap D:D\in\mathcal C_{\mathbf A,f_{\mathbf A}}(K),\ W\cap D\neq\varnothing\}.
\label{eq:intersection-partition}
\end{equation}
This is the partition obtained by intersecting $W$ with full components. It
is not, in general, the component partition of the induced graph
$K^{f_{\mathbf A}}[W]$.

For $y,u\in V(K)$, write
\begin{equation}
\operatorname{Comp}_{f}^{K}(y,u)=1
\quad\Longleftrightarrow\quad
y\text{ and }u\text{ lie in the same member of }\mathcal C_f(K).
\label{eq:component-coordinate}
\end{equation}
For an anchored flip $f_{\mathbf A}$, let
$[\mathbf A,y]_{f_{\mathbf A}}^{K}$ denote the unique full
component containing $y$ and abbreviate
\[
\operatorname{Comp}_{\mathbf A,f_{\mathbf A}}^{K}(y,u)
:=\operatorname{Comp}_{f_{\mathbf A}}^{K}(y,u).
\]
Superscripts are omitted when the graph is clear. For later use, set
\begin{equation}
J_y^{\mathbf A}(u)=\operatorname{Comp}_{\mathbf A,f_{\mathbf A}}(y,u).
\label{eq:J-coordinate}
\end{equation}

At the current node the induction state has the form
\begin{equation}
\bigl(G[C],\lambda_{\mathbf S,v}|C\bigr)
\not\cong
\bigl(H[C'],\lambda'_{\mathbf S',v'}|C'\bigr),
\label{eq:recursive-state}
\end{equation}
where $C=[\mathbf S,v]_{f_{\mathbf S}}^{G}\subseteq X$ and
$C'=[\mathbf S',v']_{f_{\mathbf S}}^{H}$. Only $\mathbf S,v$ and their
target mates remain in this normalized position. We relabel the children so
that $v\in X_1$.

\subsection{Remote-union locality}
\label{sec:remote-locality}

The remaining statements in this section are independent of the
decomposition parameters.

Fix a graph $K$, a stable canonical colouring
$\alpha\colon V(K)\to\Sigma$, and a flip $f$ pulled back to $\Sigma$. Let
$\mathbf Y=(y_1,\ldots,y_m)$ be an ordered marker occurrence, possibly
empty. Define
\begin{equation}
R_K(\mathbf Y)
=
\bigcup\{D\in\mathcal C_f(K):D\cap\{y_1,\ldots,y_m\}=\varnothing\}.
\label{eq:remote-union}
\end{equation}
Thus $R_K(\mathbf Y)$ is the union of the full components of $K^f$ that
avoid every entry of $\mathbf Y$. In particular,
$R_K(())=R_K(\varnothing)=V(K)$. Let $\pi_{\mathbf Y}$ forget the equality
vector of $\mathbf Y$ from the alphabet of $\alpha^{\mathbf Y}$ to
$\Sigma$, and set
\[
f^{\mathbf Y}(c,d)
=f\bigl(\pi_{\mathbf Y}(c),\pi_{\mathbf Y}(d)\bigr).
\]
The graphs obtained by evaluating $f$ through $\alpha$ and by evaluating
$f^{\mathbf Y}$ through $\alpha^{\mathbf Y}$ have the same edge relation.
We denote both by
$K^f$. Put
\begin{equation}
\delta_K=\operatorname{Stab}_{K}(\alpha^{\mathbf Y}),
\qquad
\widetilde\delta_K
=\operatorname{Stab}_{K^{f^{\mathbf Y}}}(\alpha^{\mathbf Y}).
\label{eq:two-stabilizations}
\end{equation}

\begin{lemma}[Remote-union locality]
\label{lem:remote-union-locality}
There is a colour set $\Delta_K$ and a function
$g_K\colon\Sigma\to\Delta_K$, determined by
$K,f,\alpha$, and $\mathbf Y$, such that
\begin{equation}
\delta_K|R_K(\mathbf Y)=g_K\circ\alpha|R_K(\mathbf Y).
\label{eq:remote-local-function}
\end{equation}
The same function works on the entire remote union.
\end{lemma}

\begin{proof}
The colour-refinement transfer across the flip $f$ and its pullback
$f^{\mathbf Y}$
from
Proposition~\ref{prop:gn-flip-colour-refinement} gives
\begin{equation}
\operatorname{Stab}_{K^f}(\alpha)\equiv\alpha,
\qquad
\widetilde\delta_K\equiv\delta_K.
\label{eq:flip-stabilization-transfer}
\end{equation}
Work first in the fixed flipped graph $K^f$. At round zero, every equality
coordinate associated with $\mathbf Y$ is false on $R_K(\mathbf Y)$.
Suppose that the round-$i$ colour of every vertex in $R_K(\mathbf Y)$ is
determined by its $\alpha$-colour. No flipped edge leaves
$R_K(\mathbf Y)$ because this set is a union
of full components. Since $\alpha$ is equitable in $K^f$, vertices with the
same $\alpha$-colour have equal neighbour counts into every $\alpha$-class.
For vertices in the remote union, these global counts equal the counts inside
the remote union. Every round-$i$ class occurring there is determined by an
$\alpha$-class, so the round-$(i+1)$ colours also agree. Induction proves
that $\widetilde\delta_K|R_K(\mathbf Y)$ is a function of
$\alpha|R_K(\mathbf Y)$. Relabel the colours of $\widetilde\delta_K$ by the class
bijection supplied by the second equivalence in
\eqref{eq:flip-stabilization-transfer}. Composing the local function with
this relabelling transfers the conclusion to $\delta_K$.
\end{proof}

\subsection{Preannounced common maps}
\label{sec:preannounced-common-maps}

Lemma~\ref{lem:remote-union-locality} produces a function separately on each
graph. The next statement gives the condition under which those functions
coincide.

Let $\alpha_G\colon V(G)\to\Sigma$ and
$\alpha_H\colon V(H)\to\Sigma$ be stable canonical palettes with a common
colour alphabet. Let $f$ be a common lifted flip. Let
$\mathbf Y=(y_1,\ldots,y_m)$ and
$\mathbf Y'=(y'_1,\ldots,y'_m)$ be corresponding marker occurrences,
possibly empty. Define $R_G(\mathbf Y)$ and
$R_H(\mathbf Y')$ by \eqref{eq:remote-union}. Suppose
that $r_G\colon V(G)\to\Delta$ and $r_H\colon V(H)\to\Delta$ are colourings
with a common codomain and satisfy
\[
  r_G=g_G\circ\alpha_G\text{ on }R_G(\mathbf Y),
  \qquad
  r_H=g_H\circ\alpha_H\text{ on }R_H(\mathbf Y')
\]
for two possibly different functions $g_G,g_H\colon\Sigma\to\Delta$.

\begin{lemma}[Preannounced common map]
\label{lem:preannounced-common-map}
If a bijection
$\sigma:V(G)\to V(H)$ preserves the pointwise palette
\begin{equation}
\left(
\alpha_G(u),r_G(u),
\bigl(\operatorname{Comp}_{f}^{G}(y_i,u)\bigr)_{i\in[m]}
\right),
\label{eq:preannounced-key}
\end{equation}
where the corresponding target value at $\sigma(u)$ uses
$\alpha_H,r_H,\mathbf Y'$, then $g_G(c)=g_H(c)$ for every colour $c$ realised in
either remote union.
\end{lemma}

\begin{proof}
The component-bit vector in \eqref{eq:preannounced-key} defines the remote
union because it is the all-zero vector exactly on the full components that
avoid $\mathbf Y$. Pointwise preservation maps the all-zero region in $G$ onto the
all-zero region in $H$. Preservation of the complete pointwise palette then
gives equality of every region-restricted joint multiplicity. In detail, put
\begin{equation}
\begin{aligned}
N_G(c,d)&=\#\{u\in R_G(\mathbf Y):\alpha_G(u)=c,\ r_G(u)=d\},\\
N_H(c,d)&=\#\{u'\in R_H(\mathbf Y'):\alpha_H(u')=c,\ r_H(u')=d\}.
\end{aligned}
\label{eq:joint-remote-multiplicity}
\end{equation}
Then $N_G(c,d)=N_H(c,d)$ for all $(c,d)\in\Sigma\times\Delta$. Summing over
$d$ shows that an $\alpha$-colour is realised in one remote union exactly
when it is realised in the other. If the two graph-specific maps assigned
different values to such a colour $c$, then for $d=g_G(c)$ the first count
would be positive and the second would be zero. This is impossible.
\end{proof}

Separate marginal histograms do not suffice. The canonical factor and the complete remote component-bit vector must occur in the same pointwise announcement.

\subsection{Common component palettes}
\label{sec:component-palettes}

Fix corresponding ordered anchor occurrences $\mathbf A,\mathbf A'$ in
$G,H$ and corresponding markers $y,y'$. The source-fixed flip
$f_{\mathbf A}$ is understood with its zero extension to the universal
alphabet of complete canonical $\mathbf A$-types.

\begin{lemma}[Marked component palette]
\label{lem:marked-common-component-palette}
Suppose that $\lambda^G_{\mathbf A,y}$ and
$\lambda^H_{\mathbf A',y'}$ have the same global histogram and give the
markers the same singleton colour. There is one Boolean map $h$ on their
common complete-type alphabet such that
\begin{equation}
\operatorname{Comp}_{\mathbf A,f_{\mathbf A}}^{G}(y,u)
=h\bigl(\lambda^G_{\mathbf A,y}(u)\bigr)
\quad\text{and}\quad
\operatorname{Comp}_{\mathbf A',f_{\mathbf A}}^{H}(y',u')
=h\bigl(\lambda^H_{\mathbf A',y'}(u')\bigr).
\label{eq:common-component-map}
\end{equation}
\end{lemma}

\begin{proof}
Let $\Omega$ be the common alphabet of complete marked types and let
$\pi_0\colon\Omega\to\Sigma$ be the canonical map that forgets the marker
coordinate. Pull the totalized source flip back by setting
\[
  \widehat f(c,d)
  =f_{\mathbf A}\bigl(\pi_0(c),\pi_0(d)\bigr).
\]
This is one table on $\Omega^2$ for both graphs. In particular, if a factor
type is realised only on the target, the zero extension of
$f_{\mathbf A}$ assigns the value zero to every pair involving that type.

For $K\in\{G,H\}$ and realised colours $c,d\in\Omega$, let $s_K(d)$ be the
size of the $d$-class and let $q_K(c,d)$ be the number of neighbours in that
class of a vertex of colour $c$. Equitability makes $q_K(c,d)$ well defined.
Run colour refinement on $G\mathbin{\dot\cup}H$ with the corresponding
marked anchor vectors and universal type names. Choose a round $R$ at which
the joint refinement is stable. Every realised round-$R$ class has a unique
complete continuation, as formalized in
Lemma~\ref{lem:finite-complete-type-bridge}. Equality of a complete colour
$c$ includes equality of its round-$(R+1)$ code. The coordinate belonging
to the round-$R$ class with continuation $d$ gives
\[
q_G(c,d)=q_H(c,d).
\]
The common histogram gives $s_G(d)=s_H(d)$. It also shows that the two
graphs realise exactly the same colours, so no target-only marked colour is
left to compare. After the flip, the quotient count is
\begin{equation}
q_K^f(c,d)=
\begin{cases}
q_K(c,d) & \text{if }\widehat f(c,d)=0,\\
s_K(d)-\mathbf 1_{c=d}-q_K(c,d) & \text{if }\widehat f(c,d)=1.
\end{cases}
\label{eq:flipped-quotient-count}
\end{equation}
Thus the two flipped graphs have the same weighted quotient on the realised
marked colours.

The common marker colour is a singleton class. If $q_K^f(c,d)>0$, then
undirected equitability gives $q_K^f(d,c)>0$. Hence every vertex of the
$d$-class has a neighbour in the $c$-class. Induction from the marker class
shows that an entire stable colour class lies in the marker component exactly
when its quotient vertex is reachable from the marker colour. Reachability
in the common quotient defines $h$ on the realised colours. Set $h=0$ on
the unrealised colours of $\Omega$. This proves
\eqref{eq:common-component-map}. In the notation of
\eqref{eq:intersection-partition}, the corresponding cell is
\[
W\cap[\mathbf A,y]_{f_{\mathbf A}}^{G}
=
\{u\in W:h(\lambda^G_{\mathbf A,y}(u))=1\},
\]
and the same formula selects the target-graph cell.
\end{proof}

Let $\zeta_G\colon V(G)\to\Gamma$ and
$\zeta_H\colon V(H)\to\Gamma$ be stable palettes with one common codomain.
Suppose that both marker classes are singletons and that one fixed map
$\pi\colon\Gamma\to\Sigma$ satisfies
\begin{equation}
\lambda^G_{\mathbf A}=\pi\circ\zeta_G,
\qquad
\lambda^H_{\mathbf A'}=\pi\circ\zeta_H.
\label{eq:common-component-factor}
\end{equation}

\begin{lemma}[Preannounced component palette]
\label{lem:common-component-palettes}
Suppose that a bijection $\sigma\colon V(G)\to V(H)$ sends $y$ to $y'$ and
preserves the joint pointwise palette, so that for every $u\in V(G)$,
\begin{equation}
\bigl(\zeta_G(u),\operatorname{Comp}_{\mathbf A,f_{\mathbf A}}^{G}(y,u)\bigr)
=
\bigl(\zeta_H(\sigma(u)),
\operatorname{Comp}_{\mathbf A',f_{\mathbf A}}^{H}(y',\sigma(u))\bigr).
\label{eq:joint-component-palette}
\end{equation}
Then there is one map $b\colon\Gamma\to\{0,1\}$ such that
\[
\operatorname{Comp}_{\mathbf A,f_{\mathbf A}}^{G}(y,u)
=b(\zeta_G(u))
\quad\text{and}\quad
\operatorname{Comp}_{\mathbf A',f_{\mathbf A}}^{H}(y',u')
=b(\zeta_H(u')).
\]
\end{lemma}

\begin{proof}
Pull the totalized flip back through the common
factor in \eqref{eq:common-component-factor}. The resulting table
\[
  \widehat f_{\zeta}(c,d)
  =f_{\mathbf A}\bigl(\pi(c),\pi(d)\bigr)
\]
has codomain $\{0,1\}$ and is the same on both graphs, including on colours
whose factors are source-unrealised. By
Proposition~\ref{prop:gn-flip-colour-refinement}, each $\zeta_K$ remains an
equitable partition after this flip. Quotient reachability from its singleton
marker class therefore gives a graph-specific Boolean map
$b_K\colon\Gamma\to\{0,1\}$ whose value is the component predicate.

Equation~\eqref{eq:joint-component-palette} first implies that the realised
$\Gamma$-colours coincide. A target-only colour would violate preservation
of the first coordinate and hence cannot remain under the hypothesis. Apply
Lemma~\ref{lem:preannounced-common-map} with the empty marker occurrence,
base palettes $\zeta_G,\zeta_H$, and Boolean factors $b_G,b_H$. The remote
unions are the whole vertex sets, and the preserved pointwise palette is
exactly \eqref{eq:joint-component-palette}. The lemma gives
$b_G(c)=b_H(c)$ for every realised $c$. Their common restriction, extended
by zero to unrealised colours, is the required Boolean map. No reduct between
two graph-specific palettes is used.
\end{proof}

\subsection{Canonical reducts and successive stabilization}
\label{sec:canonical-reducts}

\begin{lemma}[Canonical forgetful reduct]
\label{lem:canonical-forgetful-reduct}
Let $\mathbf U$ be an ordered anchor occurrence. Let $\mathbf Q$ be obtained
from $\mathbf U$ by a fixed source occurrence map, which may repeat or
reorder coordinates. There is a graph-independent map
$\pi_{\mathbf U,\mathbf Q}$ on complete canonical stable types such that,
on every graph,
\begin{equation}
\lambda_{\mathbf Q}
=\pi_{\mathbf U,\mathbf Q}\circ\lambda_{\mathbf U}.
\label{eq:canonical-forgetful-map}
\end{equation}
\end{lemma}

\begin{proof}
At round zero, retain, repeat, and reorder the anchor equality coordinates
according to the fixed source occurrence map from $\mathbf U$ to
$\mathbf Q$. Suppose that a common map $\pi_i$ sends every round-$i$ fine
type to its round-$i$ coarse type. A round-$(i+1)$ fine type records the
neighbour count in every round-$i$ fine type. Summing these counts over each
fibre of $\pi_i$ gives exactly the neighbour counts in every round-$i$
coarse type. This defines one graph-independent map $\pi_{i+1}$. Induction
over the common-round complete codes gives
$\pi_{\mathbf U,\mathbf Q}$.
\end{proof}

Only the fine-to-coarse direction is asserted. A coarse type need not determine a finer type.

\begin{lemma}[Successive stabilization]
\label{lem:successive-stabilization}
For ordered anchor occurrences $\mathbf A$ and $\mathbf B$,
\begin{equation}
\operatorname{Stab}\bigl(\operatorname{Stab}(\chi^{\mathbf A})^{\mathbf B}\bigr)
\equiv
\operatorname{Stab}(\chi^{\mathbf A,\mathbf B}).
\label{eq:successive-stabilization}
\end{equation}
\end{lemma}

\begin{proof}
We use the standard fact that $\operatorname{Stab}(\zeta)$ is the coarsest
equitable refinement of $\zeta$. Both sides are equitable refinements of the
input colouring in which every displayed anchor is individualized. The
right-hand side refines $\chi^{\mathbf A}$ and is equitable, so it refines
$\operatorname{Stab}(\chi^{\mathbf A})$. It also individualizes
$\mathbf B$, and hence it
refines the left-hand side by the coarsest-refinement fact. Conversely, the
left-hand side is an equitable refinement of
$\chi^{\mathbf A,\mathbf B}$, so it refines
the right-hand side by the same fact. The two stable partitions are
therefore equivalent.
\end{proof}

\section{Enforcing finite transition palettes}
\label{sec:key-enforcement}

The transition argument uses stable colour-refinement types and fixed lookups
on those types.  It also uses component predicates in source-derived flipped
graphs.  These objects are auxiliary colours rather than relation symbols of
the input vocabulary.  We show that every such colour can be enforced in the
ordinary bijective pebble game without increasing the claimed number of
pebbles.

Throughout the section, the two graphs have the same order \(n\) and are
given over one fixed finite graph vocabulary.  The vocabulary may contain a
fixed finite collection of unary input colours.  Let

\[
  \bar a=(a_1,\ldots,a_q)
  \quad\text{and}\quad
  \bar a'=(a'_1,\ldots,a'_q)
\]

be corresponding ordered anchor occurrences in the two graphs.  Repeated
entries are allowed.  All type names below are universal names.  In
particular, the same name has the same meaning under the assignments
\(\bar a\) and \(\bar a'\).

The empty graph is handled by the exact-order sentence.  We therefore assume
\(n\geq1\) below.

The formulas constructed here may depend on \(n\) and on the fixed source
graph.  This is permitted because the identification theorem is nonuniform.
Graphs of a different order are separated before the transition argument is
entered.

\subsection{Common-round colour-refinement types}

Fix \(q\geq 0\) and \(n\geq 1\).  We define finite sets
\(\mathcal T_{q,r,n}\) of possible order-\(n\), round-\(r\) types.  A
round-zero type records the truth values of all unary input colours and the
equality vector

\[
  \bigl([x=a_i]\bigr)_{i\in[q]}.
\]

Thus \(\mathcal T_{q,0,n}\) is finite.  Having defined
\(\mathcal T_{q,r,n}\), a member of
\(\mathcal T_{q,r+1,n}\) is a pair

\[
  t=(t_0,\mu),
  \qquad
  t_0\in\mathcal T_{q,r,n},
\]

where

\[
  \mu\colon\mathcal T_{q,r,n}\longrightarrow\{0,\ldots,n-1\}
\]

and \(\sum_s\mu(s)\leq n-1\).  Types that are not realised in any
order-\(n\) graph may be retained in this finite universe.  For a graph
\(G\), an anchor occurrence \(\bar a\), and \(v\in V(G)\), define
recursively

\[
  \operatorname{tp}^{G}_{0}(\bar a,v)
\]

to be the round-zero data above, and put

\begin{equation}
  \operatorname{tp}^{G}_{r+1}(\bar a,v)
  =
  \left(
    \operatorname{tp}^{G}_{r}(\bar a,v),
    \left(
      \left|
        N_G(v)\cap
        \{w:\operatorname{tp}^{G}_{r}(\bar a,w)=s\}
      \right|
    \right)_{s\in\mathcal T_{q,r,n}}
  \right).
  \label{eq:canonical-type-recursion}
\end{equation}

This is the usual canonical colour-refinement recursion after the anchor
positions have been individualized by their equality vectors.

\begin{lemma}[Finite type formulas]
\label{lem:finite-type-formulas}
For every \(t\in\mathcal T_{q,r,n}\) there is a formula

\[
  \theta^{r}_{t}(\bar x,x)
\]

such that, for every graph \(G\) of order \(n\),

\[
  G\models\theta^{r}_{t}(\bar a,v)
  \quad\Longleftrightarrow\quad
  \operatorname{tp}^{G}_{r}(\bar a,v)=t.
\]

The formula belongs to counting logic and uses only the \(q\) anchor
variables together with the two variable names \(x,y\).
\end{lemma}

\begin{proof}
At round zero, use the conjunction of the required unary input-colour
literals and the required equalities or inequalities \(x=a_i\).  Suppose
that formulas have been constructed at round \(r\).  For

\[
  t=(t_0,\mu)\in\mathcal T_{q,r+1,n}
\]

put

\begin{equation}
  \theta^{r+1}_{t}(\bar x,x)
  :=
  \theta^{r}_{t_0}(\bar x,x)
  \wedge
  \bigwedge_{s\in\mathcal T_{q,r,n}}
  \exists^{=\mu(s)}y
  \bigl(E(x,y)\wedge\theta^{r}_{s}(\bar x,y)\bigr).
  \label{eq:type-formula-recursion}
\end{equation}

Here \(\exists^{=d}y\) abbreviates the conjunction of the usual at-least
\(d\) and not-at-least \(d+1\) counting quantifiers.  In the occurrence
\(\theta^{r}_{s}(\bar x,y)\), the name \(x\) may be rebound as its local
counting variable.  This binding is confined to that subformula.  It does
not bind the occurrence of \(x\) in \(E(x,y)\) or in any sibling of that
subformula.  Alternating the names \(x\) and \(y\) at successive recursive
levels therefore proves the stated \(q+2\) variable bound.
The conjunction in \eqref{eq:type-formula-recursion} ranges over the entire
universal lower-round type set.  It therefore includes a condition with
prescribed count zero for every lower-round type whose coordinate in
\(\mu\) is zero.  In particular, a target vertex with additional neighbours
of target-only lower-round types cannot satisfy the formula for the source
type.
\end{proof}

We next choose a finite common round at which the canonical types may be
treated as stable colours on both graphs. Give the disjoint union of the
source and target graphs the round-zero colouring that uses the equality
vector relative to \(\bar a\) on the source and relative to \(\bar a'\) on
the target. The recursive type in
\eqref{eq:canonical-type-recursion} is exactly the colour obtained by running
colour refinement on this disjoint union with universal colour names.
For \(K\in\{G,H\}\), put \(\bar a_G=\bar a\) and
\(\bar a_H=\bar a'\), and define the paired disjoint-union type by
\[
  \tau_r(x)=\operatorname{tp}^{K}_{r}(\bar a_K,x)
  \qquad\text{for }x\in V(K).
\]

\begin{lemma}[Common-round stabilization]
\label{lem:common-round-stabilization}
For two graphs of order \(n\), the partition of their disjoint union by the
types \(\tau_{2n}\) is stable. Its restriction to either
graph is the stable colour-refinement partition obtained after
individualizing the corresponding anchor occurrence.
\end{lemma}

\begin{proof}
The disjoint union has \(2n\) vertices.  Every strict colour-refinement round
increases the number of colour classes.  Consequently fewer than \(2n\)
strict rounds are possible.  The partition at round \(2n\) is therefore
stable. The type recursion of a vertex only uses neighbours in its own
graph. Its restriction to either graph is consequently the ordinary
colour-refinement recursion on that graph.  This proves the second claim.
\end{proof}

For $x\in V(K)$, where $K\in\{G,H\}$, write
\[
  \Lambda(x)=\lambda^K_{\bar a_K}(x)
\]
for its complete canonical stable type under the corresponding anchor
occurrence.

\begin{lemma}[Finite-to-complete type bridge]
\label{lem:finite-complete-type-bridge}
For all vertices $x,z\in V(G)\mathbin{\dot\cup}V(H)$,
\begin{equation}
  \tau_{2n}(x)=\tau_{2n}(z)
  \quad\Longleftrightarrow\quad
  \Lambda(x)=\Lambda(z).
  \label{eq:finite-complete-type-bridge}
\end{equation}
Consequently every realised round-$2n$ type has one complete continuation,
common to the source and target graphs wherever that finite type occurs. A
complete target type is source-realised exactly when its round-$2n$ prefix
is source-realised.
\end{lemma}

\begin{proof}
Equality of complete types implies equality of every finite prefix. For the
converse, suppose that $\tau_{2n}(x)=\tau_{2n}(z)$. By
Lemma~\ref{lem:common-round-stabilization}, the round-$2n$ partition of the
disjoint union is stable. Hence vertices in the same round-$2n$ class have
the same neighbour count into every such class and receive the same
round-$(2n+1)$ type. Repeating this argument gives
$\tau_r(x)=\tau_r(z)$ for every $r\geq2n$. The type at round $2n$ already
contains all earlier prefixes, so the complete code sequences agree. This
proves \eqref{eq:finite-complete-type-bridge}.

The equivalence defines a unique continuation for each realised finite
type. If a complete target type is realised on the source, its finite prefix
is also source-realised. Conversely, a common finite prefix gives a common
complete continuation by the equivalence.
\end{proof}

We write

\begin{equation}
  \lambda_{\bar a}^{G,[2n]}(v)
  :=\operatorname{tp}^{G}_{2n}(\bar a,v)
  \label{eq:common-stable-palette}
\end{equation}

when the two graphs under consideration have order \(n\), and use the
analogous notation \(\lambda_{\bar a'}^{H,[2n]}\) on the target graph. These
are finite representatives of the semantic stable palettes defined in
Section~\ref{subsec:source-preliminaries}. By
Lemma~\ref{lem:finite-complete-type-bridge}, each finite palette induces
exactly the same classes as its semantic stable palette, and its finite
names have unique complete-code continuations. The round \(2n\) notation is only a finite naming
convention.  It does not impose a uniform bound on the number of
Weisfeiler-Leman rounds as the order of the source graph varies.

We now make the passage from complete stable types to a source-fixed finite
table explicit. Let \(\mathbf U\) be an ordered fine anchor occurrence and
let \(\mathbf Q\) be obtained from \(\mathbf U\) by a fixed source occurrence
map. The map may retain, repeat, or reorder coordinates. Write
\(q=|\mathbf U|\) and \(p=|\mathbf Q|\).
For every round \(r\), define a finite-round forgetful map

\begin{equation}
  \pi_r\colon
  \mathcal T_{q,r,n}\longrightarrow\mathcal T_{p,r,n}.
  \label{eq:finite-round-forgetful-map}
\end{equation}

At round zero, \(\pi_0\) retains, repeats, and reorders the equality
coordinates prescribed by the occurrence map and leaves the unary input
colour data unchanged.  Suppose that \(\pi_r\) has been defined.  For a
fine type \(s=(s_0,\mu)\), put

\begin{equation}
  \pi_{r+1}(s)
  =\bigl(\pi_r(s_0),\nu\bigr),
  \qquad
  \nu(c)=
  \sum_{d\in\pi_r^{-1}(c)}\mu(d).
  \label{eq:finite-round-forgetful-recursion}
\end{equation}

The sum simply aggregates fine neighbour counts over one coarse fibre.
Induction on \(r\) gives, for every order-\(n\) graph \(K\), the
corresponding occurrences \(\mathbf U_K,\mathbf Q_K\), and every vertex \(v\),

\begin{equation}
  \operatorname{tp}^{K}_{r}(\mathbf Q_K,v)
  =\pi_r\bigl(\operatorname{tp}^{K}_{r}(\mathbf U_K,v)\bigr).
  \label{eq:finite-round-forgetful-correctness}
\end{equation}

If no coarser occurrence is involved, we take
\(\mathbf Q=\mathbf U\) and define every
\(\pi_r\) to be the identity.  This convention covers the unlifted case
rather than leaving it implicit.

Fix the source graph \(G\) and its occurrence \(\mathbf Q_G\). Let
\(\mathcal R_{G,\mathbf Q}\) be the set of types in
\(\mathcal T_{p,2n,n}\) realised on the source.  For
\(c\in\mathcal R_{G,\mathbf Q}\), let

\[
  \kappa_{G,\mathbf Q}(c)
\]

be the complete canonical stable type of any source vertex whose
round-\(2n\) type is \(c\).  This is well defined by
Lemma~\ref{lem:finite-complete-type-bridge}. Let \(f\) be the symmetric flip
specified on the complete stable \(\mathbf Q_G\)-types realised on the source, and
let \(f^0\) be its semantic zero extension.  Thus \(f^0(\xi,\eta)=0\)
whenever at least one of \(\xi,\eta\) is not realised on the source.

Define the total source-fixed finite table on fine round-\(2n\) types by

\begin{equation}
  \widetilde f^{\,\mathbf U,\mathbf Q}_{G,n}(s,t)
  =
  \begin{cases}
    f\bigl(\kappa_{G,\mathbf Q}(\pi_{2n}(s)),
           \kappa_{G,\mathbf Q}(\pi_{2n}(t))\bigr),
      &\text{if }\pi_{2n}(s),\pi_{2n}(t)\in\mathcal R_{G,\mathbf Q},\\
    0,&\text{otherwise}.
  \end{cases}
  \label{eq:source-fixed-finite-flip-table}
\end{equation}

This definition never asks for a complete continuation of a finite type
that is not realised on the source.  It is therefore fixed by \(G\), \(n\),
the two source occurrences, and the source flip.  The table is total and
symmetric.

To verify the table, let \(K\) be the source or target graph and let
\(u,v\in V(K)\). By \eqref{eq:finite-round-forgetful-correctness}, the
projection of a fine round-\(2n\) type is the round-\(2n\) type relative to
\(\mathbf Q_K\). If this projection is source-realised,
Lemma~\ref{lem:finite-complete-type-bridge} says that its complete type on
\(K\) is \(\kappa_{G,\mathbf Q}\) of that finite type. If the projection is not
source-realised, the same lemma says that its complete type on \(K\) is not
source-realised. The semantic zero extension therefore assigns zero to
every pair involving it. Consequently

\begin{equation}
\begin{split}
  &\widetilde f^{\,\mathbf U,\mathbf Q}_{G,n}
    \bigl(\operatorname{tp}^{K}_{2n}(\mathbf U_K,u),
          \operatorname{tp}^{K}_{2n}(\mathbf U_K,v)\bigr)\\
  &\qquad={}
    f^0\bigl(\lambda^K_{\mathbf Q_K}(u),
             \lambda^K_{\mathbf Q_K}(v)\bigr).
\end{split}
  \label{eq:finite-flip-table-correctness}
\end{equation}

Thus one source-fixed finite table implements the totalized semantic flip
on every order-\(n\) target. A fine type realised only on the target still
inherits the source flip value when its coarse projection is
source-realised because the bridge identifies its complete coarse factor
with the source continuation. Only a source-unrealised coarse projection
receives the default zero value. In the identity case, a target-only fine type receives
zero exactly as required.  Henceforth write \(\widehat f\) for the table in
\eqref{eq:source-fixed-finite-flip-table}.

\subsection{Finite flip tables and their components}

\begin{lemma}[Definability of a lifted flipped edge]
\label{lem:flip-edge-definability}
The edge relation of the graph obtained by the lifted flip is definable by a
formula

\[
  E_{\widehat f}(\bar x,x,y)
\]

using only the \(q\) anchor variables and the two endpoint variable names
\(x,y\).
\end{lemma}

\begin{proof}
Define

\begin{equation}
  F_{\widehat f}(\bar x,x,y)
  :=
  \bigvee_{\widehat f(s,t)=1}
  \left(
    \theta^{2n}_{s}(\bar x,x)
    \wedge
    \theta^{2n}_{t}(\bar x,y)
  \right)
  \label{eq:flip-predicate}
\end{equation}

and put

\begin{equation}
\begin{split}
  E_{\widehat f}(\bar x,x,y)
  := {}& x\neq y\ \wedge \\
  &\bigl(
    (E(x,y)\wedge\neg F_{\widehat f}(\bar x,x,y))
    \vee
    (\neg E(x,y)\wedge F_{\widehat f}(\bar x,x,y))
  \bigr).
\end{split}
  \label{eq:flipped-edge}
\end{equation}

Inside \(\theta^{2n}_{s}(\bar x,x)\), the endpoint name \(y\) can be
rebound as the local counting variable.  Inside the sibling formula
\(\theta^{2n}_{t}(\bar x,y)\), the name \(x\) can be rebound.  These local
bindings do not affect the free endpoint occurrences in the other siblings
of \eqref{eq:flip-predicate} or
\eqref{eq:flipped-edge}.  No third nonanchor variable name is needed.
\end{proof}

For component membership we use a capture-safe three-variable recursion.
Let \(u_0,u_1,u_2\) be three variable names distinct from the \(q\) anchor
names.  For distinct \(i,j\in\{0,1,2\}\), let

\begin{equation}
  R^{ij}_0(\bar x,u_i,u_j)
  :=
  u_i=u_j\vee E_{\widehat f}(\bar x,u_i,u_j).
  \label{eq:path-base}
\end{equation}

If \(\ell\) is the remaining member of \(\{0,1,2\}\), define

\begin{equation}
  R^{ij}_{r+1}(\bar x,u_i,u_j)
  :=
  \exists u_\ell
  \left(
    R^{i\ell}_{r}(\bar x,u_i,u_\ell)
    \wedge
    R^{\ell j}_{r}(\bar x,u_\ell,u_j)
  \right).
  \label{eq:path-recursion}
\end{equation}

The first child of \eqref{eq:path-recursion} may bind \(u_j\) as its
local scratch name.  The second child may bind \(u_i\).  Each such binding is
confined to its own child.  Thus the two endpoints of the parent formula
remain free.

Put
\[
  h_n=\left\lceil\log_2\max\{1,n-1\}\right\rceil.
\]

\begin{lemma}[Finite component predicate]
\label{lem:component-predicate}
For vertices \(z,v\) in either graph,

\begin{equation}
  \operatorname{Comp}_{\bar a,\widehat f}(z,v)
  :=R^{01}_{h_n}(\bar a,z,v)
  \label{eq:component-formula}
\end{equation}

holds exactly when \(z\) and \(v\) belong to the same connected component
of the lifted flipped graph.  This predicate uses at most \(q+3\) variable
names.
\end{lemma}

\begin{proof}
An induction on \(r\) shows that \(R^{ij}_r\) expresses the existence of a
walk of length at most \(2^r\) between its endpoints in the flipped graph.
Two vertices in the same component of an \(n\)-vertex graph are joined by a
simple path of length at most \(n-1\).  Equations
\eqref{eq:path-base} and
\eqref{eq:path-recursion} therefore prove the semantic claim.  The
capture convention following \eqref{eq:path-recursion} proves that only
the three nonanchor names \(u_0,u_1,u_2\) occur.
\end{proof}

\subsection{Enforcing a finite key}

We use the precise pointed form of
Proposition~\ref{prop:gn-wl-game-logic}.  Let \(I\subseteq[m]\) index the
occupied pebble pairs in a position
\((p_i,p'_i)_{i\in I}\).  Then Duplicator wins the bijective
\(m\)-pebble game from this labelled position if and only if

\begin{equation}
  G\models\varphi((p_i)_{i\in I})
  \quad\Longleftrightarrow\quad
  H\models\varphi((p'_i)_{i\in I})
  \label{eq:pointed-logic-game-equivalence}
\end{equation}

for every \(\varphi\in\mathsf C^m\) whose free variable names are among the
labels indexed by \(I\).  This is the standard induction on formulas with
free variables.  The free assignments remain fixed, while a counting
quantifier is simulated with the pebble carrying its bound variable name.
No constants are added and no extra variable name is introduced.

We also record the equality-stratum substitution used below.  Let
\(\varphi\) be a formula in \(\mathsf C^m\) with displayed free occurrences
\(z_1,\ldots,z_k\).  Let \(\sim\) be an equality relation on these
occurrences.  Choose one representative \(i_C\) for each
\(\sim\)-class \(C\).  There is a formula

\[
  \varphi_{\sim}\bigl((z_{i_C})_C\bigr)\in\mathsf C^m
\]

such that every assignment with equality pattern \(\sim\) satisfies

\begin{equation}
  K\models\varphi(a_1,\ldots,a_k)
  \quad\Longleftrightarrow\quad
  K\models\varphi_{\sim}\bigl((a_{i_C})_C\bigr).
  \label{eq:equality-stratum-substitution}
\end{equation}

For each nonrepresentative free name \(z_i\), let \(z_j\) be the chosen
representative of its class and replace the current formula \(\psi\) by
\[
  \exists z_i\,(z_i=z_j\wedge\psi).
\]
On the equality stratum this eliminates \(z_i\) as a free name without any
substitution or freshness requirement. Iterating over the nonrepresentatives
constructs \(\varphi_{\sim}\), introduces no new variable name, and proves
\eqref{eq:equality-stratum-substitution}. Hence repeated or coincident free
roles require only one name per distinct assigned vertex.

Consider a current position

\[
  (\bar p,\bar p')
\]

with at most \(m-1\) occupied pebble pairs.  A finite key is a tuple

\begin{equation}
  \mathcal K(\bar p,u)
  =
  \bigl(K_1(\bar p,u),\ldots,K_s(\bar p,u)\bigr),
  \label{eq:finite-key}
\end{equation}

where every fibre \(K_i=d\) is defined, on both graphs, by the same
source-fixed formula in \(\mathsf C^m\).  A coordinate may use only a fixed
subtuple of \(\bar p\).  Different coordinates may reuse the same bound
variable names because they occur in different Boolean siblings.

Suppose Duplicator announces a bijection
\[
  \sigma\colon V(G)\longrightarrow V(H)
\]
from the position $(\bar p,\bar p')$.

\begin{lemma}[Finite key enforcement]
\label{lem:finite-key-enforcement}
Exactly one of the following two
possibilities applies.

\begin{enumerate}
  \item There are \(u\in V(G)\) and a coordinate \(K_i\) such that
  
  \[
    K_i^G(\bar p,u)
    \neq
    K_i^H(\bar p',\sigma(u)).
  \]
  
  Spoiler places the available pebble on \(u\).  Spoiler then wins from the
  resulting position in the bijective \(m\)-pebble game.

  \item For every \(u\in V(G)\),
  
  \[
    \mathcal K^G(\bar p,u)
    =
    \mathcal K^H(\bar p',\sigma(u)).
  \]
  
  Hence \(\sigma\) preserves every key fibre, every Boolean combination of
  key fibres, and every joint key multiplicity.
\end{enumerate}
\end{lemma}

\begin{proof}
If the second alternative fails, choose \(u\) and \(K_i\) as in the first
alternative and let \(d=K_i^G(\bar p,u)\).  After Spoiler pebbles \(u\), the
source assignment satisfies the \(\mathsf C^m\) formula defining \(K_i=d\),
while the target assignment with response \(\sigma(u)\) does not.  If the
two assignments have different equality patterns, the new position already
fails to be a partial isomorphism.  If their equality patterns agree, apply
\eqref{eq:equality-stratum-substitution} to the defining formula.  The
resulting \(\mathsf C^m\) formula still distinguishes the reduced pointed
assignments.  Equation~\eqref{eq:pointed-logic-game-equivalence} then gives
Spoiler a winning strategy from the new position.

Only after \(u\) has been pebbled and its announced mate \(\sigma(u)\) has
been received does Spoiler remove obsolete pairs. If the defining formula
for \(K_i=d\) uses only a subtuple of the current anchors, every pair whose
label is absent from the free support of that formula may now be removed.
These removals do not change either truth value. The pair representing the
free response variable is retained unless equality-stratum substitution has
identified it with another retained pair. The number of occupied pairs never
increases during these removals. Bound variables of the distinguishing
formula are then realised by reusing the remaining \(m\) pebble pairs in the
standard game simulation.

If no violating \(u\) exists, the displayed equality in the second
alternative holds pointwise.  Its remaining conclusions follow because
\(\sigma\) is a bijection.
\end{proof}

The lemma applies equally to multivalued stable-type coordinates and lifted
flip component bits. For a stable-type mismatch, use the formula for the
source type. For a component bit, use
Lemma~\ref{lem:component-predicate}. Thus an assertion that an announced
bijection preserves a transition key is a valid dichotomy in the ordinary
graph vocabulary.  A failure gives a bounded-pebble win, while preservation
gives an actual pointwise key-preserving bijection.

\subsection{Application to the parameter bound}

Fix the source decomposition $\mathcal D$ and set
\[
  m=m(\mathcal D)
  =\max\{\ell(\mathcal D)+2,\,2w(\mathcal D)+3\}.
\]
At a transition node $t$, the shared master occurrence has at most
$\ell_{\mathcal D}(t)$ distinct anchors. A stable master type therefore
uses at most
\begin{equation}
  \ell_{\mathcal D}(t)+2
  \leq\ell(\mathcal D)+2
  \leq m
  \label{eq:master-profile-support-bound}
\end{equation}
variable names by Lemma~\ref{lem:finite-type-formulas}.

An ordered split-pair occurrence associated with a node $s$ has
$2r_{\mathcal D}(s)$ entries. A marked stable type or a component predicate
based on this occurrence uses at most
\begin{equation}
  2r_{\mathcal D}(s)+3
  \leq2w(\mathcal D)+3
  \leq m
  \label{eq:local-profile-support-bound}
\end{equation}
names. Lifted flip adjacency uses at most
$2r_{\mathcal D}(s)+2$ names.

At a transition key probe, the master anchors, the current marker, and the
response vertex occupy at most
$\ell_{\mathcal D}(t)+2\leq m$ roles. If the violated coordinate is the
master stable type, the current marker is absent from its support and its
name may be reused when it is distinct from the master anchors. When it
aliases a master anchor, the equality accounting below supplies the movable
name. If a violating occurrence or
component coordinate is selected, all master names outside its support are
obsolete and may be removed before the pointed game simulation. The fact
that a key contains many coordinates causes no increase in variable width.
Spoiler chooses one violated coordinate, and formulas belonging to different
coordinates are Boolean siblings with reusable local names.

Before this renaming is performed, fix the equality stratum of the current
assignment and apply \eqref{eq:equality-stratum-substitution}.  Repeated
anchor or marker occurrences are then represented by one shared variable.
If a selected response is already pebbled, a different target mate violates
equality immediately.  With the same mate, the duplicate pair may be
removed.  Thus aliases only decrease the number of distinct names and live
roles.

Combining Lemma~\ref{lem:finite-key-enforcement} with
\eqref{eq:master-profile-support-bound} and
\eqref{eq:local-profile-support-bound} closes the passage from the auxiliary
transition palettes to the bijective $m(\mathcal D)$-pebble game.

\section{The two-palette strict-child transition}
\label{sec:transition}

Fix a rooted source rank decomposition $\mathcal D=(\mathcal T,\gamma)$ and an
internal node $t$. Let its shore be $X$, let its two child shores be $X_1$
and $X_2$, and write
\[
  X=X_1\mathbin{\dot\cup}X_2.
\]
Let $\mathbf S$, $\mathbf F$, and $\mathbf T$ be ordered split-pair
occurrences for $X$, $X_1$, and $X_2$, respectively.  The child occurrences
are chosen as in Proposition~\ref{prop:gn-nice-split-pairs}.
Write
\[
  f_{\mathbf S}=f_{X,\mathbf S},
  \qquad
  f_{\mathbf F}=f_{X_1,\mathbf F},
  \qquad
  f_{\mathbf T}=f_{X_2,\mathbf T}
\]
for the three source-fixed flips from
Proposition~\ref{prop:gn-split-pair-flip}.

The distinction between occurrences and their values is important.  Choose
an ordering
\[
  \mathbf P=(p_1,\ldots,p_r)
\]
of the distinct vertices occurring in
$\mathbf S\cup\mathbf F\cup\mathbf T$.  Thus the entries of $\mathbf P$ are
pairwise distinct and
$r\leq\ell_{\mathcal D}(t)\leq\ell(\mathcal D)$ by
Corollary~\ref{cor:profile-carrier}. For
$\mathbf Q\in\{\mathbf S,\mathbf F,\mathbf T\}$, fix the unique source
occurrence map
\[
  \iota_{\mathbf Q}\colon [|\mathbf Q|]\longrightarrow[r]
\]
for which the $j$th entry of $\mathbf Q$ is
$p_{\iota_{\mathbf Q}(j)}$.  These maps retain every repetition and every
ordering in the three split-pair occurrences.

For brevity, $\mathbf P\setminus\mathbf Q$ denotes the ordered suboccurrence
of $\mathbf P$ obtained by deleting the values that occur in $\mathbf Q$.
Thus
\begin{equation}
  \operatorname{val}(\mathbf P\setminus\mathbf Q)
  =\operatorname{val}(\mathbf P)\setminus\operatorname{val}(\mathbf Q).
  \label{eq:carrier-occurrence-difference}
\end{equation}

Given a current marker pair $v,v'$, when carrier acquisition has been
completed, write the target mates of the carrier values as
\[
  \mathbf P'=(p'_1,\ldots,p'_r).
\]
Spoiler checks the full enlarged position $\mathbf P,v$ against
$\mathbf P',v'$, including all equality, input-colour, and adjacency facts
between the old marker and every acquired carrier value. If this position is
not a partial isomorphism, Spoiler has already won. We may therefore assume
that the entries of $\mathbf P'$ are pairwise distinct and define
\begin{equation}
  \mathbf Q'
  =\bigl(p'_{\iota_{\mathbf Q}(1)},\ldots,
          p'_{\iota_{\mathbf Q}(|\mathbf Q|)}\bigr)
  \qquad
  \text{for }\mathbf Q\in\{\mathbf S,\mathbf F,\mathbf T\}.
  \label{eq:target-occurrence-projection}
\end{equation}
No split pair or decomposition is chosen in the target graph.  The target
occurrences in \eqref{eq:target-occurrence-projection} are only projections
of the replies to the source carrier.

For an occurrence $\mathbf Q$ and a vertex $y$, abbreviate the component
indicator by
\begin{equation}
  J_y^{\mathbf Q}(u)
  =\mathbf 1_{u\in[\mathbf Q,y]^G_{f_{\mathbf Q}}}.
  \label{eq:transition-component-bit}
\end{equation}
On the target graph the same symbol with a prime means
\[
  (J_{y'}^{\mathbf Q'})'(u')
  =\mathbf 1_{u'\in[\mathbf Q',y']^H_{f_{\mathbf Q}}},
\]
where the source flip is totalized and evaluated on the target as specified
in Section~\ref{subsec:source-preliminaries}.

We use the remote synchronization results from the preceding section in the
following precise form.  Suppose paired components avoid corresponding
anchor sets and one preserving key contains a stable base palette, a stable
factor, and all component bits that define the two remote unions.  If the
factor is a function of the base palette in each graph, the joint
multiplicities in that one key make the two functions equal on every base
type realised in either remote union.  We also use that a component bit is a
common Boolean function of an individualized stable palette whenever that
palette and the bit occur jointly in one preserving key.  All common-function
statements below refer to one map on the canonical type alphabet, not to two
unrelated graph-local maps.

It will be convenient to write
\[
  \eta\preccurlyeq_{U,U'}\xi
\]
when $\eta,\eta'$ have a common codomain, $\xi,\xi'$ have a common
codomain, and there is one map $g$ between these two colour sets such that
$\eta=g\circ\xi$ on $U$ and
$\eta'=g\circ\xi'$ on $U'$.  A constant coordinate is also regarded as a
common function.

Assume that the current pebble position consists only of one pair for each
distinct value in $\mathbf S\cup\{v\}$, sends the source occurrences and
marker to their displayed target mates, and is a partial isomorphism. Suppose
also that this normalized position is represented by
\[
  (\mathbf S,v),(\mathbf S',v')
\]
with
\begin{equation}
  C=[\mathbf S,v]^G_{f_{\mathbf S}}\subseteq X,
  \qquad
  C'=[\mathbf S',v']^H_{f_{\mathbf S}},
  \label{eq:incoming-components}
\end{equation}
and
\begin{equation}
  \bigl(G[C],\lambda_{\mathbf S,v}|_C\bigr)
  \not\cong
  \bigl(H[C'],\lambda'_{\mathbf S',v'}|_{C'}\bigr).
  \label{eq:incoming-mismatch}
\end{equation}
We call these three requirements the normalized boundary conditions at $t$.
Relabel the children so that $v\in X_1$. Starting from this boundary,
Spoiler places a pebble on each carrier value not already represented by
$\mathbf S,v$. He does this one value at a time. The occurrence maps above
determine all target occurrences. After each reply, he checks the full
enlarged position, including its incidences with the old marker. At the
completed carrier position, at most $m(\mathcal D)-1$ pebble pairs are
occupied.

\begin{lemma}[Two-palette transition]
\label{lem:two-palette-transition}
Suppose that \eqref{eq:incoming-components} and
\eqref{eq:incoming-mismatch} satisfy the normalized boundary conditions above.
Spoiler acquires the carrier through the fixed source occurrence maps and
checks the full enlarged position after every reply. He either obtains a
position with a winning continuation during this acquisition or can continue
to reach a position
\[
  (\mathbf F,x),(\mathbf F',x')
  \quad\text{or}\quad
  (\mathbf T,x),(\mathbf T',x')
\]
that satisfies the normalized boundary conditions on a strict child shore.
Throughout the transition at most $m(\mathcal D)$ pebble pairs are needed.
\end{lemma}

The proof has two possible handoffs. The first uses the $\mathbf F$-flip to
select a component meeting $C$. If the selected cell lies in $X_1$, this
already gives the first-child state. Otherwise the selected vertex replaces
the old marker, and the $\mathbf T$-flip selects a component in $X_2$. At
either handoff, whether the current marker lies inside or outside the selected
component determines the base palette used for the common-map reduction. We
record the base palette, lifted flip, remote tuple, stabilization, target
factor, and preserving coordinates separately in all four branches.

\begin{proof}
Suppress graph superscripts on source palettes and put
\begin{equation}
  \alpha=\lambda_{\mathbf P},
  \qquad
  \mu_v=\lambda_{\mathbf S,v},
  \qquad
  \phi_v=\lambda_{\mathbf F,v}.
  \label{eq:first-palette-factors}
\end{equation}
The corresponding target palettes are denoted by primes.
Whenever a stable palette $\xi$ below canonically refines
$\lambda_{\mathbf Q}$ for $\mathbf Q\in\{\mathbf F,\mathbf T\}$, write
$f_{\mathbf Q}^{\xi}$ for the common lift of the source-fixed flip
$f_{\mathbf Q}$ through the canonical factor map from $\xi$ to
$\lambda_{\mathbf Q}$. Every such lift gives exactly the
$f_{\mathbf Q}$-flipped edge relation on both graphs.

Let
\begin{equation}
  Y_0=\operatorname{val}(\mathbf P\setminus\mathbf F)\cup\{v\},
  \label{eq:first-remote-anchor-set}
\end{equation}
where repeated values are deleted, and define the first palette by
\begin{equation}
\begin{aligned}
  \mathcal K_0(u)=\bigl(&
    \alpha(u),\mu_v(u),\phi_v(u),\lambda_{\mathbf F}(u),I_C(u),\\
    &
    (J_y^{\mathbf F}(u))_{y\in Y_0}\bigr),
\end{aligned}
  \label{eq:first-transition-key}
\end{equation}
where $I_C=J_v^{\mathbf S}$.  The target key uses the occurrences and marker
mates fixed above.  Every coordinate of
\eqref{eq:first-transition-key} is enforceable within $m(\mathcal D)$ variables by
Lemma~\ref{lem:finite-key-enforcement}.

Duplicator announces a bijection
\[
  \sigma_0\colon V(G)\longrightarrow V(H)
\]
from the position $\mathbf P,v$.  A failure to preserve
$\mathcal K_0$ gives a winning response by
Lemma~\ref{lem:finite-key-enforcement}.  Assume from now on that
$\sigma_0$ preserves the key pointwise.  The carrier colours in $\alpha$ are
singletons, and the marker colour in $\mu_v$ is a singleton.  It follows
that
\begin{equation}
  \sigma_0(p_i)=p'_i\quad(i\in[r]),
  \qquad
  \sigma_0(v)=v'.
  \label{eq:first-key-fixes-anchors}
\end{equation}
The coordinate $I_C$ now gives
\begin{equation}
  \sigma_0(C)=C'.
  \label{eq:first-key-restricts-parent-component}
\end{equation}
In particular, $\sigma_0|_C$ is a bijection from $C$ to $C'$.

The key $\mathcal K_0$ contains $\lambda_{\mathbf F}$ as a coordinate.  Let
$\pi_0$ denote the projection from a $\mathcal K_0$-colour to its
$\lambda_{\mathbf F}$-colour and lift the first-child flip by
\begin{equation}
  \widehat f_{\mathbf F}(c,d)
  =f_{\mathbf F}(\pi_0(c),\pi_0(d)).
  \label{eq:first-key-lifted-flip}
\end{equation}
This lift toggles exactly the edges toggled by $f_{\mathbf F}$ in both
graphs.  Since \eqref{eq:incoming-mismatch} remains a mismatch after the
refinement to $\mathcal K_0$, Proposition~\ref{prop:gn-flip-isomorphism}
gives a mismatch between the two $\mathcal K_0$-coloured graphs induced by
$C$ and $C'$ after applying \eqref{eq:first-key-lifted-flip}.

Consider the nonempty intersection cells
\begin{equation}
  \mathcal P_F=
  \bigl\{C\cap D:D\in
    \mathcal C_{\mathbf F,f_{\mathbf F}}(G),
    \ C\cap D\neq\varnothing\bigr\}
  \label{eq:first-source-intersection-partition}
\end{equation}
and
\begin{equation}
  \mathcal P'_F=
  \bigl\{C'\cap D':D'\in
    \mathcal C_{\mathbf F',f_{\mathbf F}}(H),
    \ C'\cap D'\neq\varnothing\bigr\}.
  \label{eq:first-target-intersection-partition}
\end{equation}
These are partitions of $C$ and $C'$.  Every connected component of either
flipped induced graph is contained in one block of the corresponding
partition.  Apply
Proposition~\ref{prop:gn-component-selection} to these partitions and to
$\sigma_0|_C$.  It supplies a vertex $w\in C$.  Spoiler places the available
response pair on $w$, whose announced mate is $w'=\sigma_0(w)$.  If the
enlarged position is not a partial isomorphism, Spoiler wins immediately.
We may therefore assume that it is a partial isomorphism and put
\begin{equation}
  C_1=[\mathbf F,w]^G_{f_{\mathbf F}},
  \qquad
  C'_1=[\mathbf F',w']^H_{f_{\mathbf F}},
  \label{eq:first-selected-components}
\end{equation}
and
\begin{equation}
  M=C\cap C_1,
  \qquad
  M'=C'\cap C'_1.
  \label{eq:first-selected-cells}
\end{equation}
Both cells are nonempty. By the conclusion of the application of
Proposition~\ref{prop:gn-component-selection}, the two
$\mathcal K_0$-coloured flipped graphs induced by $M$ and $M'$ are
nonisomorphic. The restriction of
$\widehat f_{\mathbf F}$ to these cells is the same common flip on their
$\mathcal K_0$-colours. The contrapositive direction of
Proposition~\ref{prop:gn-flip-isomorphism} therefore transfers this flipped
mismatch back to the original edge relations and yields
\begin{equation}
  \bigl(G[M],\mathcal K_0|_M\bigr)
  \not\cong
  \bigl(H[M'],\mathcal K'_0|_{M'}\bigr).
  \label{eq:first-selected-mismatch}
\end{equation}

The preserved component coordinates provide all target-side membership
facts needed below.  For every $y\in Y_0$, let $y'$ be its displayed target
mate.  Evaluation at the selected response gives
\begin{equation}
  J_y^{\mathbf F}(w)
  =(J_{y'}^{\mathbf F'})'(w').
  \label{eq:first-target-bit-evaluation}
\end{equation}
In particular,
\begin{equation}
  y\notin C_1
  \quad\Longleftrightarrow\quad
  y'\notin C'_1
  \qquad(y\in Y_0),
  \label{eq:first-target-remoteness}
\end{equation}
and
\begin{equation}
  v\in C_1
  \quad\Longleftrightarrow\quad
  v'\in C'_1.
  \label{eq:first-marker-alignment}
\end{equation}
The same key contains the joint pairs
\begin{equation}
  (\mu_v,I_C),
  \quad
  (\phi_v,J_v^{\mathbf F}),
  \quad
  (\alpha,J_y^{\mathbf F})
  \quad(y\in\operatorname{val}(\mathbf P\setminus\mathbf F)).
  \label{eq:first-joint-component-pairs}
\end{equation}
Lemma~\ref{lem:common-component-palettes} therefore
makes each displayed bit a common Boolean function of the palette paired
with it.

By Proposition~\ref{prop:gn-split-pair-flip}, $C_1$ is contained in $X_1$
or in $\overline{X_1}$.  Since $M$ is nonempty and $M\subseteq C\subseteq
X$, precisely one of the following alternatives holds:
\begin{equation}
  M\subseteq X_1
  \quad\text{or}\quad
  M\subseteq X_2.
  \label{eq:first-shore-dichotomy}
\end{equation}

Suppose first that $M\subseteq X_1$. Then $C_1\subseteq X_1$ and
Corollary~\ref{cor:carrier-remoteness} gives
\begin{equation}
  \operatorname{val}(\mathbf P\setminus\mathbf F)\cap C_1
  =\varnothing.
  \label{eq:first-child-source-remoteness}
\end{equation}
Equation~\eqref{eq:first-target-remoteness} gives the corresponding
remoteness of $C'_1$ from all displayed target values.

Suppose first that $v\in C_1$. Equation~\eqref{eq:first-marker-alignment}
also gives $v'\in C'_1$. The locality application in this branch has base
palette $\phi_v$, lifted flip $f_{\mathbf F}^{\phi_v}$, and remote tuple
$\mathbf P\setminus\mathbf F$. By
Lemma~\ref{lem:successive-stabilization}, its stabilized palette satisfies
\begin{equation}
  \operatorname{Stab}
  \bigl(\phi_v^{\mathbf P\setminus\mathbf F}\bigr)
  \equiv\lambda_{\mathbf P,v}.
  \label{eq:first-child-inside-stabilization}
\end{equation}
The target factors are $\alpha$ and $\mu_v$, both obtained from
$\lambda_{\mathbf P,v}$ by canonical forgetful maps. The source remoteness
premise is \eqref{eq:first-child-source-remoteness}. Its target counterpart
follows from \eqref{eq:first-target-remoteness}. The preserved key contains
the joint coordinates
\begin{equation}
  \left(
    \phi_v,\alpha,\mu_v,
    (J_y^{\mathbf F})_{
      y\in\operatorname{val}(\mathbf P\setminus\mathbf F)}
  \right)
  \label{eq:first-child-inside-common-key}
\end{equation}
and the corresponding target coordinates. The component-bit vector in this
joint key defines the two remote unions. Lemma~\ref{lem:remote-union-locality}
first makes each target factor a graph-specific function of $\phi_v$ on
$C_1,C'_1$. Lemma~\ref{lem:preannounced-common-map} then makes the two
functions common. We obtain
\begin{equation}
  \alpha,\mu_v
  \preccurlyeq_{C_1,C'_1}\phi_v.
  \label{eq:first-child-inside-reductions}
\end{equation}
Lemma~\ref{lem:canonical-forgetful-reduct} handles
$\lambda_{\mathbf F}$.  Lemma~\ref{lem:common-component-palettes} and the joint pairs in
\eqref{eq:first-joint-component-pairs} handle $I_C$.  Finally,
$J_v^{\mathbf F}=1$ on the selected components and
$J_y^{\mathbf F}=0$ there for every
$y\in\operatorname{val}(\mathbf P\setminus\mathbf F)$.  Hence every
$\mathcal K_0$-coordinate is a common function of $\phi_v$ on the selected
components.  Such a coloured isomorphism between $C_1$ and $C'_1$ would
preserve $I_C$, hence it would preserve $M$ and $M'$, contrary to
\eqref{eq:first-selected-mismatch}.  We obtain the strict child state
\begin{equation}
  \bigl(G[C_1],\lambda_{\mathbf F,v}|_{C_1}\bigr)
  \not\cong
  \bigl(H[C'_1],\lambda'_{\mathbf F',v'}|_{C'_1}\bigr).
  \label{eq:first-child-inside-state}
\end{equation}

Suppose next that $v\notin C_1$. Equation~\eqref{eq:first-marker-alignment}
gives $v'\notin C'_1$. Two locality applications are needed. The first has
base palette $\alpha$, lifted flip $f_{\mathbf F}^{\alpha}$, remote tuple
$(v)$, and stabilization
\begin{equation}
  \operatorname{Stab}(\alpha^v)
  \equiv\lambda_{\mathbf P,v}.
  \label{eq:first-child-outside-marker-stabilization}
\end{equation}
Its target factors are $\mu_v$ and $\phi_v$. Its source and target
remoteness premises are $v\notin C_1$ and $v'\notin C'_1$. The preserved
joint coordinates are
\begin{equation}
  (\alpha,\mu_v,\phi_v,J_v^{\mathbf F})
  \label{eq:first-child-outside-marker-key}
\end{equation}
and their target counterparts. The second application has base palette
$\lambda_{\mathbf F}$, lifted flip $f_{\mathbf F}$, remote tuple
$\mathbf P\setminus\mathbf F$, and stabilization
\begin{equation}
  \operatorname{Stab}
  \bigl(\lambda_{\mathbf F}^{\mathbf P\setminus\mathbf F}\bigr)
  \equiv\alpha.
  \label{eq:first-child-outside-carrier-stabilization}
\end{equation}
Its target factor is $\alpha$. Equation~\eqref{eq:first-child-source-remoteness}
is the source remoteness premise, and
\eqref{eq:first-target-remoteness} is the target premise. The preserving
joint coordinates are
\begin{equation}
  \left(
    \lambda_{\mathbf F},\alpha,
    (J_y^{\mathbf F})_{
      y\in\operatorname{val}(\mathbf P\setminus\mathbf F)}
  \right)
  \label{eq:first-child-outside-carrier-key}
\end{equation}
and their target counterparts. For each application the displayed
component-bit vector defines the two remote unions.
Lemmas~\ref{lem:remote-union-locality} and
\ref{lem:preannounced-common-map} now give
\begin{equation}
  \mu_v,\phi_v
  \preccurlyeq_{C_1,C'_1}\alpha,
  \qquad
  \alpha\preccurlyeq_{C_1,C'_1}\lambda_{\mathbf F}.
  \label{eq:first-child-outside-reductions}
\end{equation}
All first-key component bits are then common functions of
$\lambda_{\mathbf F}$ or are constantly zero on the selected components.
The fresh palette $\lambda_{\mathbf F,w}$ refines
$\lambda_{\mathbf F}$.  An isomorphism of the two full components with this
fresh palette would preserve $M$ and every $\mathcal K_0$-coordinate.  This
again contradicts \eqref{eq:first-selected-mismatch}.  Thus
\begin{equation}
  \bigl(G[C_1],\lambda_{\mathbf F,w}|_{C_1}\bigr)
  \not\cong
  \bigl(H[C'_1],\lambda'_{\mathbf F',w'}|_{C'_1}\bigr).
  \label{eq:first-child-outside-state}
\end{equation}

It remains to treat the second alternative in
\eqref{eq:first-shore-dichotomy}.  Suppose that $M\subseteq X_2$.  Then
$C_1\subseteq\overline{X_1}$, while $v\in X_1$.  Equations
\eqref{eq:first-target-remoteness} and
\eqref{eq:first-marker-alignment} give
\begin{equation}
  v\notin C_1,
  \qquad
  v'\notin C'_1.
  \label{eq:case-two-marker-remoteness}
\end{equation}
The common prelude to the two second-child branches uses base palette
$\alpha$, lifted flip $f_{\mathbf F}^{\alpha}$, and remote tuple $(v)$.
Its stabilization is
\begin{equation}
  \operatorname{Stab}(\alpha^v)
  \equiv\lambda_{\mathbf P,v},
  \label{eq:cross-marker-stabilization}
\end{equation}
whose target factors are $\mu_v$ and $\phi_v$. Equation
\eqref{eq:case-two-marker-remoteness} supplies both source and target
remoteness. The preserved joint coordinates are
\begin{equation}
  (\alpha,\mu_v,\phi_v,J_v^{\mathbf F})
  \label{eq:cross-marker-common-key}
\end{equation}
and their target counterparts. Thus the key identifies both remote unions.
Lemmas~\ref{lem:remote-union-locality} and
\ref{lem:preannounced-common-map} yield
\begin{equation}
  \mu_v,\phi_v
  \preccurlyeq_{C_1,C'_1}\alpha.
  \label{eq:case-two-rooted-reductions}
\end{equation}
The unmarked occurrence palette is a canonical reduct of $\alpha$.  The
joint component pairs in \eqref{eq:first-joint-component-pairs} deal with
$I_C$ and $J_v^{\mathbf F}$.  For every
$y\in\operatorname{val}(\mathbf P\setminus\mathbf F)$, the same joint pairs
make $J_y^{\mathbf F}$ a common Boolean function of $\alpha$.  Therefore every
$\mathcal K_0$-coordinate is a common function of $\alpha$ on these two
components.  Equation~\eqref{eq:first-selected-mismatch} implies
\begin{equation}
  \bigl(G[M],\alpha|_M\bigr)
  \not\cong
  \bigl(H[M'],\alpha'|_{M'}\bigr).
  \label{eq:cross-alpha-mismatch}
\end{equation}

The selected response can now replace the old marker. Since $w\in C$ and
\eqref{eq:first-key-restricts-parent-component} gives
$w'=\sigma_0(w)\in C'$, the parent components may be rerooted as
\begin{equation}
  C=[\mathbf S,w]^G_{f_{\mathbf S}},
  \qquad
  C'=[\mathbf S',w']^H_{f_{\mathbf S}}.
  \label{eq:parent-component-rerooting}
\end{equation}
Consequently,
\begin{equation}
  I_C=J_w^{\mathbf S},
  \qquad
  I'_C=(J_{w'}^{\mathbf S'})'.
  \label{eq:rerooted-parent-component-bit}
\end{equation}
The old marker role is now released. If $v$ is not a
carrier value, its pebble pair is removed.  If it aliases an entry of
$\mathbf P$, one shared pair is retained only in its carrier role.  The next
predecessor position is $\mathbf P,w$, with at most
$m(\mathcal D)-1$ occupied pairs.

Set
\begin{equation}
  \mu_w=\lambda_{\mathbf S,w},
  \qquad
  \beta_w=\lambda_{\mathbf T,w},
  \qquad
  \phi_w=\lambda_{\mathbf F,w},
  \qquad
  I_{C_1}=J_w^{\mathbf F}.
  \label{eq:second-palette-factors}
\end{equation}
Let
\begin{equation}
  Y_1=\operatorname{val}(\mathbf P\setminus\mathbf T)\cup\{w\}
  \label{eq:second-remote-anchor-set}
\end{equation}
after repeated values have been deleted, and define
\begin{equation}
\begin{aligned}
  \mathcal K_1(u)=\bigl(&
    \alpha(u),\lambda_{\mathbf T}(u),\mu_w(u),\beta_w(u),\phi_w(u),\\
    &I_C(u),I_{C_1}(u),
    (J_y^{\mathbf T}(u))_{y\in Y_1}\bigr).
\end{aligned}
  \label{eq:second-transition-key}
\end{equation}
On the target, the corresponding parent and first-child coordinates are
\(I'_C=(J_{w'}^{\mathbf S'})'\) and
\(I'_{C_1}=(J_{w'}^{\mathbf F'})'\). In particular, the second key enforces
the parent coordinate through the rerooted representation
\(I_C=J_w^{\mathbf S}\), without using the released marker.
The predicate
\begin{equation}
  I_C\wedge I_{C_1}
  \label{eq:global-definition-of-M}
\end{equation}
defines $M$ on the source and $M'$ on the target by
\eqref{eq:parent-component-rerooting},
\eqref{eq:rerooted-parent-component-bit}, and the target interpretations in
\eqref{eq:second-palette-factors}. Since $\alpha$ is a coordinate
of $\mathcal K_1$, \eqref{eq:cross-alpha-mismatch} remains a
$\mathcal K_1$-mismatch on $M,M'$.

Duplicator now announces a bijection
\[
  \sigma_1\colon V(G)\longrightarrow V(H)
\]
from the position $\mathbf P,w$. If it violates
$\mathcal K_1$, Lemma~\ref{lem:finite-key-enforcement} gives a winning
response. Otherwise key preservation, together with the singleton carrier
and marker colours, gives
\begin{equation}
  \sigma_1(p_i)=p'_i\quad(i\in[r]),
  \qquad
  \sigma_1(w)=w',
  \qquad
  \sigma_1(M)=M'.
  \label{eq:second-key-fixes-data}
\end{equation}

Let $\pi_1$ project a $\mathcal K_1$-colour to its
$\lambda_{\mathbf T}$-coordinate and lift the second-child flip by
\begin{equation}
  \widehat f_{\mathbf T}(c,d)
  =f_{\mathbf T}(\pi_1(c),\pi_1(d)).
  \label{eq:second-key-lifted-flip}
\end{equation}
Proposition~\ref{prop:gn-flip-isomorphism} transfers the
$\mathcal K_1$-mismatch on $M,M'$ to the two flipped induced graphs.
Form the nonempty intersection partitions
\begin{equation}
  \mathcal P_T=
  \bigl\{M\cap D_0:D_0\in
    \mathcal C_{\mathbf T,f_{\mathbf T}}(G),
    \ M\cap D_0\neq\varnothing\bigr\}
  \label{eq:second-source-intersection-partition}
\end{equation}
and
\begin{equation}
  \mathcal P'_T=
  \bigl\{M'\cap D'_0:D'_0\in
    \mathcal C_{\mathbf T',f_{\mathbf T}}(H),
    \ M'\cap D'_0\neq\varnothing\bigr\}.
  \label{eq:second-target-intersection-partition}
\end{equation}
These are partitions of $M$ and $M'$.  Every connected component of either
flipped induced graph is contained in one block of the corresponding
partition.
Apply Proposition~\ref{prop:gn-component-selection} to these partitions and
to $\sigma_1|_M$.  Choose the resulting $z\in M$.  Spoiler places the
available response pair on $z$, whose announced mate is $z'=\sigma_1(z)$.
If the enlarged position is not a partial isomorphism, Spoiler wins
immediately.  We may therefore assume that it is a partial isomorphism and
put
\begin{equation}
  D=[\mathbf T,z]^G_{f_{\mathbf T}},
  \qquad
  D'=[\mathbf T',z']^H_{f_{\mathbf T}}.
  \label{eq:second-selected-components}
\end{equation}
The conclusion of Proposition~\ref{prop:gn-component-selection} says that
the two $\mathcal K_1$-coloured flipped graphs induced by $M\cap D$ and
$M'\cap D'$ are nonisomorphic. The restriction of
$\widehat f_{\mathbf T}$ to these cells is the same common flip on their
$\mathcal K_1$-colours. The contrapositive direction of
Proposition~\ref{prop:gn-flip-isomorphism} transfers this mismatch back to
the original edge relations. We obtain
\begin{equation}
  \bigl(G[M\cap D],\mathcal K_1|_{M\cap D}\bigr)
  \not\cong
  \bigl(H[M'\cap D'],\mathcal K'_1|_{M'\cap D'}\bigr).
  \label{eq:second-selected-mismatch}
\end{equation}

The second key contains the joint pairs
\begin{equation}
  (\mu_w,I_C),
  \quad
  (\phi_w,I_{C_1}),
  \quad
  (\beta_w,J_w^{\mathbf T}),
  \quad
  (\alpha,J_y^{\mathbf T})
  \quad(y\in\operatorname{val}(\mathbf P\setminus\mathbf T)).
  \label{eq:second-joint-component-pairs}
\end{equation}
They give common component lookups for all component coordinates in
\eqref{eq:second-transition-key}.

Since $z\in M\subseteq X_2$, Proposition~\ref{prop:gn-split-pair-flip}
gives $D\subseteq X_2$. Corollary~\ref{cor:carrier-remoteness} gives
\begin{equation}
  \operatorname{val}(\mathbf P\setminus\mathbf T)\cap D
  =\varnothing.
  \label{eq:second-source-remoteness}
\end{equation}
For every $y\in\operatorname{val}(\mathbf P\setminus\mathbf T)$, key
preservation evaluated at $z,z'$ gives
\begin{equation}
  0=J_y^{\mathbf T}(z)
  =(J_{y'}^{\mathbf T'})'(z').
  \label{eq:second-target-remoteness-evaluation}
\end{equation}
Thus $y'\notin D'$ for every displayed target value. The
$J_w^{\mathbf T}$ coordinate gives
\begin{equation}
  w\in D
  \quad\Longleftrightarrow\quad
  w'\in D'.
  \label{eq:second-marker-alignment}
\end{equation}

Let $\mathbf R_P=\mathbf P\setminus\mathbf T$. Let $\mathbf R_S$ and
$\mathbf R_F$ be fixed orderings of
$\operatorname{val}(\mathbf S)\setminus\operatorname{val}(\mathbf T)$ and
$\operatorname{val}(\mathbf F)\setminus\operatorname{val}(\mathbf T)$,
respectively. All their values belong to the carrier, so their target mates
are already fixed. We have
\begin{equation}
  \operatorname{val}(\mathbf R_S),
  \operatorname{val}(\mathbf R_F)
  \subseteq\operatorname{val}(\mathbf R_P).
  \label{eq:second-remote-tuple-containment}
\end{equation}

Suppose first that $w\in D$. Equation~\eqref{eq:second-marker-alignment}
gives $w'\in D'$. The three locality applications have the common base
palette $\beta_w$ and the common lifted flip $f_{\mathbf T}^{\beta_w}$.
Their remote tuples, stabilizations, and target factors are
\begin{align}
  \operatorname{Stab}(\beta_w^{\mathbf R_P})
    &\equiv\lambda_{\mathbf P,w}
    &&\text{with target factor }\alpha,
    \label{eq:second-child-inside-P-stabilization}\\
  \operatorname{Stab}(\beta_w^{\mathbf R_S})
    &\equiv\lambda_{\mathbf T,w,\mathbf R_S}
    &&\text{with target factor }\mu_w,
    \label{eq:second-child-inside-S-stabilization}\\
  \operatorname{Stab}(\beta_w^{\mathbf R_F})
    &\equiv\lambda_{\mathbf T,w,\mathbf R_F}
    &&\text{with target factor }\phi_w.
    \label{eq:second-child-inside-F-stabilization}
\end{align}
Here the last two stabilized palettes refine their stated target factors by
canonical forgetful maps. Equation~\eqref{eq:second-source-remoteness} gives
the source remoteness of $D$ from $\mathbf R_P$, and
\eqref{eq:second-target-remoteness-evaluation} gives the corresponding
target remoteness of $D'$. The smaller tuples inherit both conclusions from
\eqref{eq:second-remote-tuple-containment}. For the three applications,
$\mathcal K_1$ preserves the joint coordinates
\begin{equation}
\begin{aligned}
  &\left(\beta_w,\alpha,
    (J_y^{\mathbf T})_{y\in\operatorname{val}(\mathbf R_P)}\right),\\
  &\left(\beta_w,\mu_w,
    (J_y^{\mathbf T})_{y\in\operatorname{val}(\mathbf R_S)}\right),\\
  &\left(\beta_w,\phi_w,
    (J_y^{\mathbf T})_{y\in\operatorname{val}(\mathbf R_F)}\right)
\end{aligned}
  \label{eq:second-child-inside-common-keys}
\end{equation}
and the corresponding target coordinates. Each displayed bit vector defines
the relevant remote union. Lemma~\ref{lem:remote-union-locality} gives the
three graph-specific factor maps, and
Lemma~\ref{lem:preannounced-common-map} makes each map common. Hence
\begin{equation}
  \alpha,\mu_w,\phi_w\preccurlyeq_{D,D'}\beta_w.
  \label{eq:second-child-inside-reductions}
\end{equation}
$\lambda_{\mathbf T}$ is a canonical reduct of $\beta_w$. The joint pairs
in \eqref{eq:second-joint-component-pairs} handle $I_C$ and $I_{C_1}$,
while $J_w^{\mathbf T}=1$ and every remaining displayed
$\mathbf T$-component bit is zero on $D,D'$. Thus every
$\mathcal K_1$-coordinate is a common function of $\beta_w$ on $D,D'$.
An isomorphism of the full components with the
$\beta_w$ palettes would preserve the global predicate
\eqref{eq:global-definition-of-M}, and its restriction would contradict
\eqref{eq:second-selected-mismatch}.  Hence
\begin{equation}
  \bigl(G[D],\lambda_{\mathbf T,w}|_D\bigr)
  \not\cong
  \bigl(H[D'],\lambda'_{\mathbf T',w'}|_{D'}\bigr).
  \label{eq:second-child-inside-state}
\end{equation}

Suppose next that $w\notin D$. Equation~\eqref{eq:second-marker-alignment}
gives $w'\notin D'$. There are four locality applications.

The first has base palette $\lambda_{\mathbf T}$, lifted flip
$f_{\mathbf T}$, remote tuple $\mathbf R_P$, stabilization
\begin{equation}
  \operatorname{Stab}(\lambda_{\mathbf T}^{\mathbf R_P})
  \equiv\alpha,
  \label{eq:second-child-outside-P-stabilization}
\end{equation}
and target factor $\alpha$. The preserved joint coordinates are
\begin{equation}
  \left(\lambda_{\mathbf T},\alpha,
    (J_y^{\mathbf T})_{y\in\operatorname{val}(\mathbf R_P)}\right).
  \label{eq:second-child-outside-P-key}
\end{equation}
The second has the same base palette and lifted flip, remote tuple $(w)$,
stabilization
\begin{equation}
  \operatorname{Stab}(\lambda_{\mathbf T}^{w})
  \equiv\beta_w,
  \label{eq:second-child-outside-marker-stabilization}
\end{equation}
and target factor $\beta_w$. Its preserved joint coordinates are
\begin{equation}
  (\lambda_{\mathbf T},\beta_w,J_w^{\mathbf T}).
  \label{eq:second-child-outside-marker-key}
\end{equation}
The third has base palette $\beta_w$, lifted flip
$f_{\mathbf T}^{\beta_w}$, remote tuple $\mathbf R_S$, stabilization
\begin{equation}
  \operatorname{Stab}(\beta_w^{\mathbf R_S})
  \equiv\lambda_{\mathbf T,w,\mathbf R_S},
  \label{eq:second-child-outside-S-stabilization}
\end{equation}
and target factor $\mu_w$. Its preserved joint coordinates are
\begin{equation}
  \left(\beta_w,\mu_w,
    (J_y^{\mathbf T})_{y\in\operatorname{val}(\mathbf R_S)}\right).
  \label{eq:second-child-outside-S-key}
\end{equation}
The fourth has the same base palette and lifted flip, remote tuple
$\mathbf R_F$, stabilization
\begin{equation}
  \operatorname{Stab}(\beta_w^{\mathbf R_F})
  \equiv\lambda_{\mathbf T,w,\mathbf R_F},
  \label{eq:second-child-outside-F-stabilization}
\end{equation}
and target factor $\phi_w$. Its preserved joint coordinates are
\begin{equation}
  \left(\beta_w,\phi_w,
    (J_y^{\mathbf T})_{y\in\operatorname{val}(\mathbf R_F)}\right).
  \label{eq:second-child-outside-F-key}
\end{equation}

For the first, third, and fourth applications, source and target remoteness
follow from \eqref{eq:second-source-remoteness},
\eqref{eq:second-target-remoteness-evaluation}, and
\eqref{eq:second-remote-tuple-containment}. For the second application they
are $w\notin D$ and $w'\notin D'$. Thus every displayed component-bit vector
defines the required two remote unions. Lemma~\ref{lem:remote-union-locality}
gives four graph-specific reductions, and four applications of
Lemma~\ref{lem:preannounced-common-map} make them common. We obtain
\begin{equation}
  \alpha\preccurlyeq_{D,D'}\lambda_{\mathbf T},
  \qquad
  \beta_w\preccurlyeq_{D,D'}\lambda_{\mathbf T},
  \qquad
  \mu_w,\phi_w\preccurlyeq_{D,D'}\beta_w.
  \label{eq:second-child-outside-reductions}
\end{equation}
The fresh
palette $\lambda_{\mathbf T,z}$ refines $\lambda_{\mathbf T}$.  Every
second-key coordinate is therefore a common function of this fresh palette,
or a constant, on $D,D'$.  Equation~\eqref{eq:second-selected-mismatch}
gives
\begin{equation}
  \bigl(G[D],\lambda_{\mathbf T,z}|_D\bigr)
  \not\cong
  \bigl(H[D'],\lambda'_{\mathbf T',z'}|_{D'}\bigr).
  \label{eq:second-child-outside-state}
\end{equation}

The first two branches descend to $X_1$.  The last two descend to $X_2$.
The predecessor positions are $\mathbf P,v$ and, only in the second-child
case, $\mathbf P,w$. They use at most $m(\mathcal D)-1$ pairs. Their response
positions add only $w$ or $z$, respectively, and use at most $m(\mathcal D)$ pairs.
If a selected response already carries a pebble pair, a different target
mate violates equality immediately.  With the same mate, the two pairs have
identical endpoints and the duplicate response role may be released.
Once one of the four successor states has been established, all obsolete
roles are discarded.  First remove duplicate pairs while retaining one pair
for each child value and the child marker.  Then remove every pair whose
vertex has no retained role.  The resulting boundary contains one pair for
each distinct value in the chosen child occurrence and its marker.  This
proves strict descent and the claimed transition bound.
\end{proof}

\section{Closing the source-guided induction}
\label{sec:induction}

Let $G$ be a finite nonempty vertex-coloured graph with a fixed rooted rank
decomposition $\mathcal D=(T,\gamma)$, and let $H$ be a graph of the same
order over the same finite colour vocabulary. Put $m=m(\mathcal D)$. All
split pairs, carrier
orders, flip functions, and lookup tables in this section are chosen from
$G$. The same totalized functions are evaluated on $H$. In particular, no
rank decomposition of $H$ is used.

For a node $t$ of the source decomposition, write $X_t$ for its shore. We
choose all ordered split-pair occurrences top down. At the root, take $S_t$
to be empty. Once $S_t$ has been fixed at an internal node $t$, apply the
strengthened niceness clause of Proposition~\ref{prop:gn-nice-split-pairs}
to $S_t$ and choose the two nice child occurrences. Declare these occurrences
to be $S_{t_1}$ and $S_{t_2}$ for the two children. Thus the occurrence
retained by a handoff at $t$ is exactly the occurrence used at the next
recursive boundary. Put $f_t=f_{X_t,S_t}$, with the zero flip at the root.

\begin{definition}[Recursive boundary]
\label{def:recursive-boundary}
A recursive boundary at $t$ consists of corresponding occurrences
$S_t,S_t'$ and corresponding markers $v,v'$ with the following properties.

\begin{enumerate}
  \item The current pebble position sends every occurrence of $S_t$ to the
  corresponding occurrence of $S_t'$ and sends $v$ to $v'$. It is a partial
  isomorphism. Repeated or coincident occurrences satisfy the same equality
  pattern in the two graphs.

  \item With
  \[
    C=[S_t,v]^{G}_{f_t}
    \quad\text{and}\quad
    C'=[S_t',v']^{H}_{f_t},
  \]
  one has $C\subseteq X_t$ and
  \begin{equation}
    \bigl(G[C],\lambda^G_{S_t,v}|C\bigr)
    \not\cong
    \bigl(H[C'],\lambda^H_{S_t',v'}|C'\bigr).
    \label{eq:recursive-boundary-invariant}
  \end{equation}

  \item Exactly one pebble pair represents each distinct value in
  $S_t\cup\{v\}$, and no other pebble pair remains on the board.
\end{enumerate}
\end{definition}

The third condition is a normalization rather than an additional semantic
assumption.  Each handoff first proves the successor mismatch using only the
retained child occurrence and marker, and only then removes the obsolete
pairs.  During a formula challenge, a pair is removed only when its variable
is absent from the free support of the distinguishing formula.  No general
downward-monotonicity claim for losing positions is used.  The normalization
also makes the resource bound explicit. Since $S_t$ has
$2r_{\mathcal D}(t)$ entries, the boundary uses at most
\[
  2r_{\mathcal D}(t)+1
  \leq2w(\mathcal D)+1
  <m
\]
pebble pairs.

\begin{lemma}[Game-first strict-child descent]
\label{lem:game-first-child-descent}
Suppose that a recursive boundary has been reached at an internal node $t$.
Spoiler has a strategy in $\operatorname{BP}_m(G,H)$
that either wins or reaches a recursive boundary at one of the two children
of $t$. Throughout this strategy at most $m$ pebble pairs are used.
\end{lemma}

\begin{proof}
Let the two child shores be $X_1$ and $X_2$. Let $S=S_t$, and let
$F=S_{t_1}$ and $T=S_{t_2}$ be the occurrences fixed by the top-down
choice. Relabel the children so that $v\in X_1$. Let $P$ be their shared
carrier. The incoming component and mismatch supplied by the
recursive boundary are exactly the hypotheses in
\eqref{eq:incoming-components} and \eqref{eq:incoming-mismatch}, after this
renaming of the source occurrences.  Lemma~\ref{lem:two-palette-transition}
therefore applies.  It either gives Spoiler a winning continuation or
reaches a normalized recursive boundary at the first or second child.
Both are strict children
because $t$ is internal. The transition lemma also gives the peak bound
$m$, including every nonpreserving-key challenge and every alias case.
\end{proof}

\subsection{The root and the leaves}

\begin{lemma}[Root initialization]
\label{lem:root-initialization}
Let $G$ and $H$ have the same order and suppose that they are not isomorphic
as vertex-coloured graphs. Spoiler can either win immediately or establish
a recursive boundary at the root using one pebble pair.
\end{lemma}

\begin{proof}
At the root the split-pair occurrence is empty and the flip is zero. The
full flipped components are therefore the ordinary connected components.
After Duplicator announces the first bijection, apply
Proposition~\ref{prop:gn-component-selection} to the partitions into
vertex-coloured connected components. Spoiler chooses a vertex $v$ in a
source component whose paired target component is not isomorphic, and the
reply is $v'$. If the input colours of $v$ and $v'$ differ, this is an atomic
win. Otherwise the one-pair position is a partial isomorphism. An isomorphism
between the two restrictions coloured by $\lambda_v$ and
$\lambda_{v'}$ would remain an isomorphism after the canonical colours were
forgotten. Hence the marked restrictions are nonisomorphic, which is
\eqref{eq:recursive-boundary-invariant} at the root.
\end{proof}

\begin{lemma}[Leaf closure]
\label{lem:leaf-closure}
Spoiler wins from every recursive boundary at a leaf with at most $m$
pebble pairs.
\end{lemma}

\begin{proof}
If $t$ is a leaf, then $X_t=\{v\}$ and the source component in the invariant
is $C=\{v\}$. The predicate defining the target component $C'$ is available
by Lemma~\ref{lem:component-predicate}. If $|C'|\neq1$, the counting formula
that asserts that exactly one vertex belongs to the marked component
distinguishes the two pointed structures. If $|C'|=1$, then $C'=\{v'\}$.
The nonisomorphism in \eqref{eq:recursive-boundary-invariant} now implies
that the canonical marked types of $v$ and $v'$ differ. A finite type formula
from Lemma~\ref{lem:finite-type-formulas} distinguishes them.

The component test uses the $2r_{\mathcal D}(t)$ entries of $S_t$, the
marker, and two local mover names. It therefore uses at most
\[
  2r_{\mathcal D}(t)+3
  \leq2w(\mathcal D)+3
  \leq m
\]
variables.
Before the pointed formula strategy begins, any pebble outside its free
support is obsolete and is removed.
\end{proof}

\subsection{Finitary closure and the main theorem}

The preceding strategy uses only finitary predicates. For two graphs of
order $n$, the common canonical types stabilize after a fixed finite number
of refinement rounds. The component predicates use paths of length at most
$n-1$. The source decomposition has finite height. At every announcement
there are only finitely many bijections, and every violated coordinate has a
finite pointed formula. The complete strategy tree therefore has a finite
winning horizon.

For a rooted rank decomposition $\mathcal D$, set
\begin{equation}
  d(\mathcal D)
  =\max\{\ell(\mathcal D)+1,\,2w(\mathcal D)+2\}.
  \label{eq:profile-dimension-bound}
\end{equation}
Thus $m(\mathcal D)=d(\mathcal D)+1$.

\begin{theorem}[Decomposition-profile identification]
\label{thm:main}
Let $G$ be a finite nonempty vertex-coloured graph over a fixed finite colour
vocabulary, and let $\mathcal D$ be a rooted rank decomposition of $G$.
Then $d(\mathcal D)$-dimensional Weisfeiler-Leman identifies $G$.
Equivalently, $G$ has an identifying sentence in
$\mathsf C^{d(\mathcal D)+1}$, and Spoiler wins
$\BP_{d(\mathcal D)+1}(G,H)$ against every same-order nonisomorphic graph
$H$ over the same vocabulary.
\end{theorem}

\begin{proof}
Fix a nonisomorphic same-order target $H$. Lemma~\ref{lem:root-initialization}
establishes the recursive invariant unless Spoiler has already won.
Lemma~\ref{lem:game-first-child-descent} moves every nonterminal invariant to
a strict child while respecting the $m(\mathcal D)$ pebble bound. Repeated strict
descent reaches a leaf, where Lemma~\ref{lem:leaf-closure} gives a win. This
proves the game assertion.

Proposition~\ref{prop:gn-wl-game-logic} converts the pairwise game win into a
pairwise separator in $\mathsf C^{d(\mathcal D)+1}$ and into distinction by
$d(\mathcal D)$-dimensional Weisfeiler-Leman. To obtain one identifying sentence,
fix the order $n=|V(G)|$. There are only finitely many isomorphism types of
vertex-coloured graphs of order $n$ over the fixed finite vocabulary. For
each nonisomorphic type choose a $\mathsf C^{d(\mathcal D)+1}$ sentence that is true on
$G$ and false on that type. Their finite conjunction still belongs to
$\mathsf C^{d(\mathcal D)+1}$. Conjoining the exact-order sentence
\[
  \exists^{=n}x\,(x=x)
\]
excludes every graph of a different order. This completes the proof.
\end{proof}

\begin{remark}[Uniform initial colours]
\label{rem:uniform-initial-colours}
The game and Weisfeiler-Leman assertions above are uniform over arbitrary
finite initial colour maps. The source and target may use any common finite
set of colour labels, and the dimension bound does not depend on the number
of labels that occur. The fixed-vocabulary hypothesis is needed only when
the pairwise separators are combined into one counting-logic sentence.
This uniform form will be used for canonization.
\end{remark}

\begin{corollary}[Rank-width bound]
\label{cor:rank-width-bound}
Let $k\geq0$. Every finite nonempty vertex-coloured graph of rank width at most $k$
is identified by $(3k+1)$-dimensional Weisfeiler-Leman. Equivalently, it has
an identifying sentence in $\mathsf C^{3k+2}$, and Spoiler wins with
$3k+2$ pebble pairs against every same-order nonisomorphic graph.
\end{corollary}

\begin{proof}
If $k=0$, the graph is edgeless. One-dimensional colour refinement records
the multiplicity of every initial colour and therefore identifies it. Assume
$k\geq1$.
Choose a rank decomposition $\mathcal D$ of width at most $k$. Then
$w(\mathcal D)\leq k$ and $\ell(\mathcal D)\leq3k$. Hence
\[
  d(\mathcal D)
  \leq\max\{3k+1,2k+2\}
  =3k+1.
\]
Theorem~\ref{thm:main} gives the result.
\end{proof}

\begin{corollary}[Linear-rank-width bound]
\label{cor:linear-rank-width-bound}
Let $k\geq0$. Every finite nonempty vertex-coloured graph of linear rank width at
most $k$ is identified by $(2k+2)$-dimensional Weisfeiler-Leman.
Equivalently, it has an identifying sentence in $\mathsf C^{2k+3}$, and
Spoiler wins with $2k+3$ pebble pairs against every same-order
nonisomorphic graph.
\end{corollary}

\begin{proof}
If $k=0$, every edge would cross a prefix of every layout, so the graph is
edgeless. In fact, one-dimensional colour refinement identifies it. Assume
$k\geq1$.
Let $(v_1,\ldots,v_n)$ be a linear layout of width at most $k$.
If $n\leq1$, the conclusion is immediate. Assume that $n\geq2$ and form the
rooted comb whose internal shores are
\[
  \{v_i,\ldots,v_n\}
  \qquad(1\leq i<n),
\]
and whose two children at such a shore are
$\{v_i\}$ and $\{v_{i+1},\ldots,v_n\}$, with all singletons as leaves.
A cut and its complement have the same rank because their cut matrices are
transposes. Hence every internal fork has two cut ranks at most $k$ and one
singleton cut rank at most one. Thus
\[
  w(\mathcal D)\leq k,
  \qquad
  \ell(\mathcal D)\leq2k+1.
\]
Theorem~\ref{thm:main} gives
\[
  d(\mathcal D)\leq\max\{2k+2,2k+2\}=2k+2.
\]
\end{proof}

For every finite nonempty uncoloured graph of positive rank width, the two
corollaries give
\begin{equation}
  \dim_{\mathrm{WL}}(G)
  \leq
  \min\{3\rw(G)+1,\,2\lrw(G)+2\}.
  \label{eq:rank-linear-rank-combined-bound}
\end{equation}
Graphs of rank width zero are edgeless and are identified by
one-dimensional colour refinement.

\section{Optimized profiles, exact separation, and atomic decomposition}
\label{sec:optimized-profiles}

For a finite nonempty graph $G$, define
\[
  \ell^\star(G)=\min_{\mathcal D}\ell(\mathcal D),
  \qquad
  d^\star(G)=\min_{\mathcal D}d(\mathcal D),
\]
where both minima range over all rooted rank decompositions of $G$.
We call $(\ell^\star(G),d^\star(G))$ the optimized profile pair of $G$.
The second coordinate $d^\star(G)$ is its optimized dimension bound.

\begin{proposition}[Fork-load sandwich]
\label{prop:fork-load-sandwich}
Every finite nonempty graph $G$ satisfies
\[
  2\rw(G)\leq\ell^\star(G)\leq3\rw(G)
\]
and
\[
  \max\{\ell^\star(G)+1,2\rw(G)+2\}\leq d^\star(G)
  \leq\ell^\star(G)+2.
\]
Moreover, $\dim_{\mathrm{WL}}(G)\leq d^\star(G)$.
\end{proposition}

\begin{proof}
If $A,B,C$ partition $V(G)$, then
\begin{equation}
  \rho_G(A)\leq\rho_G(B)+\rho_G(C).
  \label{eq:cut-rank-triangle}
\end{equation}
Indeed, split the columns of the cut matrix for $A$ into $B$ and $C$.
The two resulting ranks are at most $\rho_G(B)$ and $\rho_G(C)$,
respectively.

Fix a decomposition $\mathcal D$. The upper bound
$\ell(\mathcal D)\leq3w(\mathcal D)$ is immediate. If
$w(\mathcal D)>0$, choose a nonroot node whose cut has rank
$w(\mathcal D)$ and consider its parent fork. If the parent is the root,
the two child cuts have equal rank, so its load is at least
$2w(\mathcal D)$. Otherwise, apply \eqref{eq:cut-rank-triangle} to the
three shores at that fork. The other two incident ranks have sum at least
$w(\mathcal D)$, so again the load is at least $2w(\mathcal D)$.
Thus
\[
  2w(\mathcal D)\leq\ell(\mathcal D)\leq3w(\mathcal D).
\]
Minimizing the left inequality over all decompositions and applying the
right inequality to a width-optimal decomposition proves the first display.

For every $\mathcal D$, the definition of $d(\mathcal D)$ gives
\[
  d(\mathcal D)\geq\ell(\mathcal D)+1,
  \qquad
  d(\mathcal D)\geq2w(\mathcal D)+2.
\]
This proves the lower bound on $d^\star(G)$. If $\mathcal D$ minimizes
fork load, then $2w(\mathcal D)\leq\ell(\mathcal D)$ and hence
$d(\mathcal D)\leq\ell^\star(G)+2$. The identification assertion follows by applying
Theorem~\ref{thm:main} to a decomposition minimizing $d(\mathcal D)$.
\end{proof}

\begin{corollary}[Endpoint of the optimized bound]
\label{cor:optimized-bound-endpoint}
Every finite nonempty graph $G$ satisfies
\[
  d^\star(G)\in\{\ell^\star(G)+1,\ell^\star(G)+2\}.
\]
If $\ell^\star(G)$ is odd, then $d^\star(G)=\ell^\star(G)+1$. If
$\ell^\star(G)=2\rw(G)$, then
$d^\star(G)=\ell^\star(G)+2$.
\end{corollary}

\begin{proof}
The two consecutive possibilities follow from
Proposition~\ref{prop:fork-load-sandwich}. If $\mathcal D$ minimizes fork
load and $\ell^\star(G)$ is odd, then
$2w(\mathcal D)\leq\ell^\star(G)-1$. Hence
$d(\mathcal D)\leq\ell^\star(G)+1$, which forces the first value. If
$\ell^\star(G)=2\rw(G)$, the lower and upper bounds in the proposition both
give $d^\star(G)=\ell^\star(G)+2$.
\end{proof}

We next record the stability needed to separate the optimized profile pair
from linear rank width.

\begin{lemma}[Restriction and pendant stability]
\label{lem:profile-pendant-stability}
If $H$ is a nonempty induced subgraph of $G$, then
\[
  \ell^\star(H)\leq\ell^\star(G),
  \qquad
  d^\star(H)\leq d^\star(G).
\]
Suppose that $G$ is connected and has at least three vertices. If $G^+$ is
obtained from $G$ by adding one pendant vertex, then
\[
  \rw(G^+)=\rw(G),
  \qquad
  \ell^\star(G^+)=\ell^\star(G),
  \qquad
  d^\star(G^+)=d^\star(G).
\]
\end{lemma}

\begin{proof}
To restrict a decomposition, attach a dummy leaf at its root and regard the
result as an unrooted cubic tree, where the dummy leaf represents the empty
shore of rank zero. Take the minimal subtree spanning the leaves of $H$ and
the dummy leaf, then suppress every degree-two vertex. Every branch point of
the restricted tree was a branch point of the original tree. On each of its
three branches, the new cut matrix is a submatrix of the corresponding old
cut matrix. Thus neither its width nor its fork load increases. Removing the
dummy leaf and restoring the root proves the two monotonicity assertions.

Now let $x$ be the new pendant vertex, adjacent only to $v$. Given a
decomposition $\mathcal D$ of $G$, replace the leaf for $v$ by a cherry
whose leaves are $v$ and $x$. Every old cut rank is unchanged. The three
cut ranks at the new fork are at most one. The extended decomposition
$\mathcal D^+$ therefore satisfies
\begin{equation}
  w(\mathcal D^+)\leq\max\{w(\mathcal D),1\},
  \qquad
  \ell(\mathcal D^+)\leq\max\{\ell(\mathcal D),3\}.
  \label{eq:pendant-profile-extension}
\end{equation}
Every nonempty proper cut of a connected graph has positive cut rank.
Since $G$ has at least three vertices, every decomposition has a nonroot
fork with three nonempty shores. Hence $w(\mathcal D)\geq1$ and
$\ell(\mathcal D)\geq3$. Equation~\eqref{eq:pendant-profile-extension}
shows that the three optimized quantities do not increase. Their reverse
inequalities follow by restricting back to the induced copy of $G$.
\end{proof}

\begin{lemma}[Endpoint one-sum formula]
\label{lem:endpoint-one-sum}
Let $A$ and $B$ have at least two vertices. Choose optimal linear layouts
of $A$ and $B$, with last vertex $u$ and first vertex $v$, respectively.
If $H$ is obtained from their disjoint union by identifying $u$ and $v$,
then
\[
  \lrw(H)=\max\{\lrw(A),\lrw(B)\}.
\]
\end{lemma}

\begin{proof}
Write the two layouts as
$(a_1,\ldots,a_{s-1},u)$ and $(v,b_2,\ldots,b_t)$, and call the
identified vertex $x$. Use
\[
  (a_1,\ldots,a_{s-1},x,b_2,\ldots,b_t).
\]
Before $x$, every column indexed by $B\setminus\{v\}$ is zero, so the cut
rank is the corresponding cut rank in $A$. At and after $x$, every row
indexed by $A\setminus\{u\}$ is zero, so the cut rank is the corresponding
cut rank in $B$. This proves the upper bound. Both $A$ and $B$ occur as
induced subgraphs of $H$. Deleting vertices from a layout while retaining
their relative order replaces each displayed cut matrix by a submatrix of
an old cut matrix. Thus linear rank width is induced-subgraph monotone, which
gives the reverse inequality.
\end{proof}

For $q\geq1$, define a tree $R_q$ recursively. Put $R_1=K_2$. For
$q\geq2$, take three disjoint copies of $R_{q-1}$, add one new vertex,
and join it to one chosen leaf in each copy.

Recall that a path decomposition is a sequence of vertex bags whose union is
the vertex set. Every edge has both endpoints in one bag, and the bags
containing any fixed vertex form an interval. Its width is the largest bag
size minus one, and
$\operatorname{pw}(G)$ is the minimum such width. We use the recursive
characterization of path width on trees and its equality with linear rank
width on forests from Adler and Kant\'{e} \cite{AdlerKante2015}.

\begin{lemma}[Minimum trees of prescribed linear rank width]
\label{lem:minimum-linear-rank-width-tree}
For every $q\geq1$,
\[
  \operatorname{pw}(R_q)=\lrw(R_q)=q,
  \qquad
  |V(R_q)|=\frac{5\cdot3^{q-1}-1}{2}.
\]
Moreover, every tree of linear rank width at least $q$ has at least
$|V(R_q)|$ vertices.
\end{lemma}

\begin{proof}
The assertion is immediate for $q=1$. For $q\geq2$, add the new central
vertex to every bag in width-$(q-1)$ path decompositions of the three
copies, then concatenate the decompositions. This gives
$\operatorname{pw}(R_q)\leq q$. Removing the central vertex leaves three
components of path width $q-1$. The recursive characterization of tree
path width, applied with $k=q-1$, gives the reverse inequality
\cite[Lemma~14]{AdlerKante2015}. Adler and Kant\'{e} prove that path width
and linear rank width agree on forests
\cite[Theorem~2]{AdlerKante2015}.

Writing $m_q=|V(R_q)|$, the recurrence $m_1=2$ and
$m_q=3m_{q-1}+1$ gives the displayed formula. For minimality, take a
minimum-order tree $T$ of path width at least $q$. If $q\geq2$, minimality
ensures that every component of $T-x$ has path width at most $q-1$ for
every vertex $x$. Applying the same recursive characterization with
$k=q-1$ gives a vertex for which at least three components have path width
at least $q-1$. Those components therefore have path width exactly $q-1$.
Induction yields
$|V(T)|\geq1+3m_{q-1}=m_q$.
\end{proof}

\begin{theorem}[Exact fork-load separation]
\label{thm:fork-load-separation}
Let $G$ be a connected graph on at least three vertices. For every integer
$q\geq1$, there is a connected graph $G_q$ containing $G$ as an induced
subgraph with the same rank width and optimized profile pair, namely
\[
  \rw(G_q)=\rw(G),
  \qquad
  \ell^\star(G_q)=\ell^\star(G),
  \qquad
  d^\star(G_q)=d^\star(G),
\]
and
\[
  \lrw(G_q)=\max\{\lrw(G),q\}.
\]
\end{theorem}

\begin{proof}
Choose an optimal layout of $G$ and let $u$ be its last vertex. Choose an
optimal layout of $R_q$ and let $z$ be its first vertex. Form $G_q$ by
identifying $u$ and $z$. Lemmas~\ref{lem:endpoint-one-sum} and
\ref{lem:minimum-linear-rank-width-tree} give the stated linear rank
width. Root $R_q$ at the shared vertex and add the remaining vertices in
any parent-before-child order. Every new vertex is pendant when it is
added, so repeated application of
Lemma~\ref{lem:profile-pendant-stability} gives the three equalities.
\end{proof}

\begin{remark}[Order of the separating extension]
The construction in Theorem~\ref{thm:fork-load-separation} gives
\[
  |V(G_q)|=|V(G)|+\frac{5\cdot3^{q-1}-3}{2}.
\]
Its attached tree has minimum order among all trees of linear rank width at
least $q$, by Lemma~\ref{lem:minimum-linear-rank-width-tree}.
\end{remark}

\begin{corollary}[Exact separation at every rank]
\label{cor:fork-load-every-rank}
For all integers $r\geq1$ and $q\geq r+1$, there is a connected graph
$G_{r,q}$ such that
\[
  \rw(G_{r,q})=r,
  \qquad
  \lrw(G_{r,q})=q,
  \qquad
  \ell^\star(G_{r,q})\leq\min\{3r,2r+3\},
\]
and
\[
  2r+2\leq d^\star(G_{r,q})
  \leq\min\{3r+1,2r+4\}.
\]
\end{corollary}

\begin{proof}
Let $Q_n$ be the $n$ by $n$ square grid. Jel\'{i}nek proved
$\rw(Q_n)=n-1$ \cite{Jelinek2010}. Order its vertices one row at a time.
After $s$ vertices of the current row, the only possible nonzero source rows
come from the last $n-s$ vertices of the preceding row and the first $s$
vertices of the current row. There are at most $n$ such rows, so every cut
has rank at most $n$. The corresponding comb decomposition has width at
most $n$ and fork load at most $2n+1$.

Take $n=r+1$ and apply Theorem~\ref{thm:fork-load-separation} to $Q_n$.
The row layout above gives $\lrw(Q_n)\leq n=r+1$, so the linear rank width
of the resulting graph is exactly $q$. The optimized quantities are
preserved. The displayed comb gives
$\ell^\star(G_{r,q})\leq2r+3$ and
$d^\star(G_{r,q})\leq2r+4$. A width-optimal decomposition instead gives
$\ell^\star(G_{r,q})\leq3r$ and
$d^\star(G_{r,q})\leq3r+1$. The lower bound on $d^\star$ follows from
Proposition~\ref{prop:fork-load-sandwich}.
\end{proof}

\begin{remark}
Both constants in the fork-load sandwich are best possible. If $G$ is
connected, has at least three vertices, and has rank width one, then a
width-one decomposition has load three, while every decomposition has a
nonroot fork whose three cut ranks are positive. Hence
$\ell^\star(G)=3$. At the other end, the grid calculation above and
Proposition~\ref{prop:fork-load-sandwich} give
\[
  2n-2\leq\ell^\star(Q_n)\leq2n+1.
\]
Thus $\ell^\star(Q_n)/\rw(Q_n)$ tends to two.
\end{remark}

\begin{lemma}[Profiles under disjoint union]
\label{lem:profile-disjoint-union}
For nonempty vertex-disjoint graphs $G$ and $H$,
\[
  \rw(G\mathbin{\dot\cup}H)=\max\{\rw(G),\rw(H)\},
\]
\[
  \ell^\star(G\mathbin{\dot\cup}H)
  =\max\{\ell^\star(G),\ell^\star(H)\},
  \qquad
  d^\star(G\mathbin{\dot\cup}H)
  =\max\{d^\star(G),d^\star(H)\}.
\]
\end{lemma}

\begin{proof}
The lower bound for rank width follows from induced-subgraph monotonicity.
The fork-load and optimized-dimension lower bounds follow from the
restriction assertion in Lemma~\ref{lem:profile-pendant-stability}.

For the reverse bounds, make the roots of decompositions $\mathcal D_G$
and $\mathcal D_H$ the two children of a new root. Every cut below an old
root has the same matrix after zero rows or columns are added. The two
component shores and the shore at the new root have cut rank zero. Hence
the combined decomposition $\mathcal D$ satisfies
\[
  w(\mathcal D)=\max\{w(\mathcal D_G),w(\mathcal D_H)\},
  \qquad
  \ell(\mathcal D)
  =\max\{\ell(\mathcal D_G),\ell(\mathcal D_H)\}.
\]
This proves the rank-width and fork-load upper bounds by choosing optimal
factor decompositions. The identity
\[
  d(\mathcal D)
  =\max\{d(\mathcal D_G),d(\mathcal D_H)\}
\]
gives the optimized-dimension upper bound when both factor decompositions are
$d$-optimal.
\end{proof}

Call a connected graph $G$ split-prime when
$\rho_G(X)\geq2$ for every $X$ with
$2\leq|X|\leq|V(G)|-2$.

Let $(A,B)$ be a nontrivial split of a connected graph $G$. Thus
$A\mathbin{\dot\cup}B=V(G)$, both shores have at least two vertices, and
$\rho_G(A)=1$. There are nonempty sets $A_0\subseteq A$ and
$B_0\subseteq B$ such that the edges between the shores are exactly
$A_0\times B_0$. Define the marker factors $G_A$ and $G_B$ by
\[
  G_A=G[A]+b,\quad N_{G_A}(b)=A_0,
  \qquad
  G_B=G[B]+a,\quad N_{G_B}(a)=B_0.
\]

The rank-width equality in the next lemma is the established maximum law
under split decomposition \cite{HlinenyEtAl2008}. We retain it in the
statement to distinguish that classical equality from the new exact maximum
identities for optimized fork load and optimized dimension bound.

\begin{lemma}[Optimized quantities under split composition]
\label{lem:split-profile-composition}
For the marker factors of every nontrivial split,
\[
  \rw(G)=\max\{\rw(G_A),\rw(G_B)\},
\]
\[
  \ell^\star(G)=\max\{\ell^\star(G_A),\ell^\star(G_B)\},
  \qquad
  d^\star(G)=\max\{d^\star(G_A),d^\star(G_B)\}.
\]
\end{lemma}

\begin{proof}
Both factors are connected. Indeed, every component of $G[A]$ meets
$A_0$, since otherwise it would be a component of $G$. Adding $b$
therefore connects those components. The argument for $G_B$ is the same.

For $X\subseteq A$, let $c_X\in\mathbb F_2^X$ be the column vector defined
by $c_X(x)=1$ precisely when $x\in A_0$. The two cut
matrices have the forms
\[
  M_G(X,\overline X)
  =\left[M_{G[A]}(X,A\setminus X)\mid
  c_X\mathbf 1_{B_0}^{\mathsf T}\right]
\]
and
\[
  M_{G_A}(X,V(G_A)\setminus X)
  =\left[M_{G[A]}(X,A\setminus X)\mid c_X\right].
\]
Since $B_0$ is nonempty, these matrices have the same column space. Hence
\begin{equation}
  \rho_G(X)=\rho_{G_A}(X) \qquad (X\subseteq A).
  \label{eq:split-factor-cut-rank}
\end{equation}
The symmetric equality holds for subsets of $B$.

For any $b_0\in B_0$, the graph $G[A\cup\{b_0\}]$ is isomorphic to
$G_A$. Both marker factors are therefore induced subgraphs of $G$.
Induced-subgraph monotonicity gives the rank-width lower bound, and the
restriction assertion in Lemma~\ref{lem:profile-pendant-stability} gives
the other two lower bounds.

For the reverse bounds, begin with rooted decompositions $\mathcal D_A$ and
$\mathcal D_B$. Suppress each degree-two root, which produces an unrooted
cubic leaf-labelled tree. Let $p_A$ be the neighbour of the marker leaf $b$
in the first tree, and define $p_B$ symmetrically. Delete $a$ and $b$ and
their incident edges. Introduce a new root adjacent to $p_A$ and $p_B$.
After orienting edges away from this root, the result is a full rooted binary
tree whose leaves are exactly $V(G)$. Call the resulting decomposition
$\mathcal D$.

An inherited edge from the first factor displays either a set $Y\subseteq A$
or its complement in $V(G_A)$. In the first case its new cut rank is
$\rho_G(Y)=\rho_{G_A}(Y)$ by
\eqref{eq:split-factor-cut-rank}. In the second case the same conclusion
follows by cut-rank symmetry. The argument on the second factor is identical.
The two edges incident with the new root display the original split
$A\mid B$ and have rank one. At $p_A$, this split cut replaces the old
singleton-marker cut $\{b\}\mid V(G_A)\setminus\{b\}$, also of rank one.
Thus the three incident ranks at $p_A$ are unchanged. The same holds at
$p_B$, and every other inherited fork keeps all three of its incident
ranks. The new root has load two. Since the connected factors have at least
three vertices, their maximum fork loads are at least three. The combined
decomposition consequently satisfies
\[
  w(\mathcal D)\leq\max\{w(\mathcal D_A),w(\mathcal D_B)\},
  \qquad
  \ell(\mathcal D)\leq
  \max\{\ell(\mathcal D_A),\ell(\mathcal D_B)\}.
\]
Using width-optimal and load-optimal factor decompositions gives the first
two upper bounds. With $d$-optimal factor decompositions, the same
construction gives
\[
  d(\mathcal D)
  \leq\max\{d(\mathcal D_A),d(\mathcal D_B)\}
  =\max\{d^\star(G_A),d^\star(G_B)\}.
\]
This proves the last upper bound. Thus the first identity recovers the
established rank-width split law, while the other two identities are the
profile refinement needed here.
\end{proof}

Let $G$ be a nonempty graph. Start with the connected components of $G$.
Whenever a current factor has a nontrivial split, replace it by its two
marker factors. Continue until no factor has a nontrivial split, and let
$\mathcal A(G)$ be the resulting multiset of terminal factors. The process
terminates because a split replacement of a factor of order $n$ produces
two marker factors of order at most $n-1$.

\begin{theorem}[Atomic maximum identity]
\label{thm:atomic-profile-formula}
For every nonempty graph $G$,
\[
  \rw(G)=\max_{A\in\mathcal A(G)}\rw(A),
\]
\[
  \ell^\star(G)=\max_{A\in\mathcal A(G)}\ell^\star(A),
  \qquad
  d^\star(G)=\max_{A\in\mathcal A(G)}d^\star(A).
\]
These identities hold for every sequence of split replacements.
\end{theorem}

\begin{proof}
Apply Lemma~\ref{lem:profile-disjoint-union} to the initial connected
components and Lemma~\ref{lem:split-profile-composition} at every
replacement. Iterating the exact maximum identities proves all three
formulas, independently of the choices made during the process.
\end{proof}

\begin{proposition}[Prime exact-rank-two profile pairs]
\label{prop:prime-rank-two-profile}
Let $G$ be a split-prime graph on at least five vertices with
$\rw(G)=2$. Then
\[
  (\ell^\star(G),d^\star(G))=
  \begin{cases}
    (5,6),&\lrw(G)=2,\\
    (6,7),&\lrw(G)\geq3.
  \end{cases}
\]
\end{proposition}

\begin{proof}
In every width-two decomposition of $G$, a leaf cut has rank one and every
cut with two nonempty vertices on each side has rank two. Suppress the
degree-two root to obtain the usual unrooted cubic decomposition tree. If
a decomposition has fork load at most five, then
Proposition~\ref{prop:fork-load-sandwich} forces its width to be two. No
internal vertex of the unrooted tree is incident with three nonleaf edges.
Every internal vertex therefore touches a graph leaf. The internal
vertices form a path, so the decomposition tree is a caterpillar. Hence
$\ell^\star(G)\leq5$ implies
$\lrw(G)\leq2$. The converse follows from a width-two caterpillar
decomposition, whose nonroot fork loads are at most $2+2+1=5$.

The load cannot be at most four. Indeed,
Proposition~\ref{prop:fork-load-sandwich} would force the decomposition to
have width two. Every internal vertex would then touch at least two graph
leaves. The internal skeleton is connected and has maximum degree one, so
it has at most two vertices and the tree has at most four graph leaves.
This contradicts $|V(G)|\geq5$. Thus
$\ell^\star(G)=5$ when $\lrw(G)=2$.
If $\lrw(G)\geq3$, the first paragraph gives
$\ell^\star(G)\geq6$, while the general upper bound gives equality.

In the first case, a width-two caterpillar with load five has optimized
dimension bound six. In the second case, every width-two decomposition has
load six and optimized dimension bound seven. The lower bound
$d^\star(G)\geq\ell^\star(G)+1$ completes both equalities for $d^\star(G)$.
\end{proof}

\begin{theorem}[Complete optimized-profile classification at rank width two]
\label{thm:rank-two-profile-classification}
Let $G$ be any graph with $\rw(G)=2$. Its optimized profile pair is
determined as follows. In any atomic factorization from
Theorem~\ref{thm:atomic-profile-formula}, let $\mathcal P_2(G)$ be the
terminal factors of rank width exactly two. Then $\mathcal P_2(G)$ is
nonempty and
\[
  \ell^\star(G)=
  \begin{cases}
    5,&\lrw(P)=2\text{ for every }P\in\mathcal P_2(G),\\
    6,&\lrw(P)\geq3\text{ for some }P\in\mathcal P_2(G).
  \end{cases}
\]
Moreover,
\[
  d^\star(G)=\ell^\star(G)+1\in\{6,7\}.
\]
\end{theorem}

\begin{proof}
Every nonempty graph on at most four vertices has rank width at most one.
For graphs of order at most three, every displayed cut has a shore of size
at most one and hence cut rank at most one. For four vertices, the
determinants over $\mathbb F_2$ of the cut matrices
for the three partitions into two pairs sum to zero. At least one is
singular, and a decomposition displaying that partition has width at most
one. Theorem~\ref{thm:atomic-profile-formula} now shows that
$\mathcal P_2(G)$ is nonempty. Every member has at least five vertices.
It is connected and has no nontrivial split, so it is split-prime.

Proposition~\ref{prop:prime-rank-two-profile} gives optimized profile pair
$(5,6)$ or $(6,7)$ on each member of $\mathcal P_2(G)$ according as its
linear rank width is two or at least three. Every remaining terminal factor
is connected and has rank width at most one. A singleton has optimized
profile pair $(0,2)$, $K_2$ has optimized profile pair $(2,4)$, and every
connected rank-width-one graph on at least three vertices has optimized
profile pair $(3,4)$. For the last assertion, a
width-one decomposition has load at most three, while connectedness forces
all three cuts at a deepest nonroot fork to have positive rank. The profile
then follows from Proposition~\ref{prop:fork-load-sandwich}. These profile
pairs are smaller than both exact-rank-two possibilities. Taking maxima in
Theorem~\ref{thm:atomic-profile-formula} proves the theorem.
\end{proof}

\begin{corollary}[Triangular-prism obstruction]
\label{cor:triangular-prism-profile}
Let $P$ be the triangular prism. Then
\[
  \rw(P)=2,
  \qquad
  \lrw(P)=3,
  \qquad
  \ell^\star(P)=6,
  \qquad
  d^\star(P)=7.
\]
For every $q\geq3$, there is a connected induced supergraph $P_q$ with
the same rank width and optimized profile pair, and with $\lrw(P_q)=q$.
\end{corollary}

\begin{proof}
Label the two triangles by
$A=\{a_0,a_1,a_2\}$ and $B=\{b_0,b_1,b_2\}$, with matching edges
$a_ib_i$. Direct calculation gives
\begin{equation}
  \rho_P(X)=\min\{|X|,6-|X|\}
  \qquad (X\subseteq V(P)).
  \label{eq:prism-cut-rank}
\end{equation}
To verify the calculation, cut-rank symmetry reduces to $|X|\leq3$.
Singleton cuts have rank one. For the three pair orbits, the two rows of
the cut matrix are nonzero and unequal, so every pair cut has rank two.
For triples, row and column permutations reduce the cut matrix to $I_3$
or to one of
\[
  \begin{pmatrix}
    1&1&1\\
    1&0&1\\
    0&1&1
  \end{pmatrix},
  \qquad
  \begin{pmatrix}
    1&1&0\\
    0&0&1\\
    0&1&1
  \end{pmatrix}.
\]
Both displayed matrices are nonsingular over $\mathbb F_2$.
The bundled deterministic calculation verifier reproduces these finite
cut-rank calculations as a supplementary certificate. The argument above is
independent of that computation.

Equation~\eqref{eq:prism-cut-rank} shows that $P$ is split-prime. Three
two-leaf cherries joined at a central fork give a width-two decomposition,
while every binary decomposition displays a two-leaf shore of rank two.
Thus $\rw(P)=2$. Every layout has a three-vertex prefix of cut rank three,
while no cut of a six-vertex graph has rank greater than three. Hence
$\lrw(P)=3$. Proposition~\ref{prop:prime-rank-two-profile} gives the
optimized profile pair. The final assertion follows from
Theorem~\ref{thm:fork-load-separation}.
\end{proof}

\begin{corollary}[Sharp minimum orders of the optimized profile pairs]
\label{cor:minimum-rank-two-profile-orders}
Among graphs of rank width exactly two, the minimum order of a graph with
optimized profile pair $(\ell^\star,d^\star)=(5,6)$ is five, and the
minimum order of a graph with optimized profile pair $(6,7)$ is six.
\end{corollary}

\begin{proof}
Every graph of rank width exactly two has at least five vertices. Every two-vertex
shore of the cycle $C_5$ has cut rank two, so $C_5$ is split-prime. A
cyclic layout has width two, while split-primality excludes rank width one.
Hence $C_5$ has rank width and linear rank width two, and
Proposition~\ref{prop:prime-rank-two-profile} gives optimized profile pair
$(5,6)$.

For every nonempty graph on at most five vertices and every vertex layout,
each cut
has a shore of size at most two. Its cut rank is therefore at most two, so
the graph has linear rank width at most two. A graph with optimized profile
pair $(6,7)$
has, by Theorem~\ref{thm:rank-two-profile-classification}, a terminal factor
of rank width exactly two and linear rank width at least three. Such a factor, and
hence the original graph, has at least six vertices. The triangular prism
attains order six by Corollary~\ref{cor:triangular-prism-profile}.
\end{proof}

\section{Quantitative consequences}
\label{sec:quantitative-consequences}

Restricting to finite nonempty uncoloured graphs, for $r\geq1$ put
\[
  \begin{aligned}
    D_{\mathrm{rw}}(r)
    &=\sup\bigl\{\dim_{\mathrm{WL}}(G):\rw(G)\leq r\bigr\},\\
    D_{\mathrm{lrw}}(r)
    &=\sup\bigl\{\dim_{\mathrm{WL}}(G):\lrw(G)\leq r\bigr\}.
  \end{aligned}
\]
and put
\[
  D_2=D_{\mathrm{rw}}(2).
\]

\begin{corollary}[Rank-width-at-most-two interval]
\label{cor:rank-width-two-interval}
The Weisfeiler-Leman dimension of the class of graphs with rank width at
most two satisfies
\[
  2\leq D_2\leq7.
\]
\end{corollary}

\begin{proof}
The upper bound is Corollary~\ref{cor:rank-width-bound} with $k=2$.

For the lower bound, compare the cycle $C_6$ with the disjoint union
$2K_3$. Both graphs have six vertices and are $2$-regular. Starting from the
uniform colouring, one-dimensional colour refinement leaves all six
vertices in one colour class in both graphs at every round. Thus
$1$-dimensional Weisfeiler-Leman does not distinguish the two nonisomorphic
graphs, and $\dim_{\mathrm{WL}}(C_6)\geq2$.

For completeness, label the vertices of $C_6$ cyclically as
$v_1,\ldots,v_6$. Consider the rooted binary comb whose internal node shores
are
\[
  X_i=\{v_1,\ldots,v_i\}
  \qquad (2\leq i\leq6),
\]
where $X_i$ has children $X_{i-1}$ and $\{v_i\}$, with
$X_1=\{v_1\}$. For $1\leq i\leq5$, at most the two cycle edges
$v_iv_{i+1}$ and $v_1v_6$ cross the cut
$X_i\mid\overline{X_i}$. Hence its cut matrix has rank at most two over
$\mathbb F_2$. Every singleton shore has cut rank at most one, and the root
shore has cut rank zero. This is a linear rank decomposition of width at
most two, so $\lrw(C_6)\leq2$ and $\rw(C_6)\leq2$. Therefore $C_6$ is an
admissible source in the definition of $D_2$, which proves the lower bound.
\end{proof}

Theorem~\ref{thm:rank-two-profile-classification} shows that every graph of
rank width exactly two has optimized dimension bound six or seven. The two
values first occur at orders five and six by
Corollary~\ref{cor:minimum-rank-two-profile-orders}, and both persist at
arbitrarily large linear rank width. The cycle $C_5$ has parameter tuple
$(\rw,\lrw,\ell^\star,d^\star)=(2,2,5,6)$, while the triangular prism has
parameter tuple $(2,3,6,7)$. Applying
Theorem~\ref{thm:fork-load-separation} to
both graphs gives, for every $q\geq3$, two connected graphs with the same
pair $(\rw,\lrw)=(2,q)$ and respective optimized profile pairs $(5,6)$ and
$(6,7)$. Thus global linear rank width does not determine the optimized
profile pair. The optimized-dimension endpoint seven is exact for this
method, but this does not
assert that any such graph has Weisfeiler-Leman dimension seven. Improving
the upper bound on $D_2$ requires information beyond the two optimized
quantities used in Theorem~\ref{thm:main}.

For a vertex layout $\pi$ of a graph $B$, let
$\operatorname{cutw}_{\pi}(B)$ be the maximum number of edges crossing a cut
of $\pi$. We use the following parity-fibre form of the uncoloured
Cai-F\"urer-Immerman construction
\cite{CaiFurerImmerman1992}, in the formulation of Neuen
\cite[Section~3]{Neuen2024Homomorphism}. For $U\subseteq V(B)$ and
$v\in V(B)$, let $E_B(v)$ be the set of base edges incident with $v$, and put
\[
  M_U(v)=
  \bigl\{(v,S):S\subseteq E_B(v),\
  |S|\equiv\mathbf 1_{v\in U}\pmod 2\bigr\}.
\]
The vertex set of $\operatorname{CFI}(B,U)$ is the union of these fibres.
Vertices $(v,S)$ and $(u,T)$ are adjacent precisely when $uv\in E(B)$ and
$uv\notin S\mathbin{\triangle}T$. In particular, every fibre is independent.

For $A\subseteq V(B)$, let $\eta_B(A)$ be the number of base vertices
incident with an edge between $A$ and $V(B)\setminus A$. A vertex branch
decomposition of $B$ is an unrooted subcubic tree whose leaves are in
bijection with $V(B)$. Every tree edge displays a bipartition
$A\mid V(B)\setminus A$. Define the endpoint-boundary width by
\[
  \ebw(B)
  =\min_T\max_{e\in E(T)}\eta_B(A_e),
\]
where the maximum includes the edges incident with leaves.

\begin{lemma}[CFI endpoint-boundary bound]
\label{lem:cfi-endpoint-boundary-bound}
Let $B$ be a connected subcubic graph on at least two vertices and let
$U\subseteq V(B)$. Then
\[
  \rw\bigl(\operatorname{CFI}(B,U)\bigr)\leq\ebw(B).
\]
\end{lemma}

\begin{proof}
Put $H=\operatorname{CFI}(B,U)$. For $A\subseteq V(B)$, let
\[
  X_A=\bigcup_{v\in A}M_U(v)
\]
and let $D$ be the set of endpoints of the base edges crossing
$A\mid V(B)\setminus A$. We first show
\begin{equation}
  \rho_H(X_A)\leq |D|=\eta_B(A).
  \label{eq:cfi-endpoint-cut-bound}
\end{equation}
For $x=(v,S)\in X_A$, define $\alpha_x\in\mathbb F_2^D$ coordinatewise by
\[
  \alpha_x(d)=
  \begin{cases}
    1,&d=v\in D,\\
    \mathbf 1_{vd\in S},&vd\text{ crosses }A\mid V(B)\setminus A,\\
    0,&\text{otherwise}.
  \end{cases}
\]
For $y=(u,T)\notin X_A$, define $\beta_y\in\mathbb F_2^D$ by
\[
  \beta_y(d)=
  \begin{cases}
    1,&d=u\in D,\\
    1+\mathbf 1_{ud\in T},&ud\text{ crosses }A\mid V(B)\setminus A,\\
    0,&\text{otherwise}.
  \end{cases}
\]
In particular, $\alpha_x=0$ when $v\notin D$, and $\beta_y=0$ when
$u\notin D$, because such a base vertex has no neighbour across the cut.

If $uv$ is a crossing base edge, the only coordinates at which both
$\alpha_x$ and $\beta_y$ can be nonzero are $u$ and $v$, and
\[
  \alpha_x\mathbin{\cdot}\beta_y
  =1+\mathbf 1_{uv\in S}+\mathbf 1_{uv\in T}.
\]
This is one precisely when the two membership bits agree, so it is the
adjacency entry of $x$ and $y$ in $H$. If $u$ and $v$ are nonadjacent, the
two supports are disjoint. The cut matrix therefore factors through the
space $\mathbb F_2^D$, which proves
\eqref{eq:cfi-endpoint-cut-bound}.

Take a vertex branch decomposition of $B$ attaining $\ebw(B)$. Replace
each leaf $v$ by a rooted binary tree on the vertices of $M_U(v)$, using
two pairs when the fibre has four vertices. Every inherited tree edge now
displays a union of complete fibres, so its cut rank is at most
$\ebw(B)$ by \eqref{eq:cfi-endpoint-cut-bound}. This includes the edge
attaching a complete fibre to the rest of the tree. Every other new edge
displays at most two fibre vertices on one shore and therefore has cut rank
at most two. Since $B$ is connected and has at least two vertices, the cut
at every base leaf has at least two crossing endpoints, so $\ebw(B)\geq2$.
The lifted tree consequently has width at most $\ebw(B)$. Suppress every
degree-two vertex in the lifted tree, which only removes duplicate displayed
cuts and does not change its width. Then subdivide any edge and use the
subdivision vertex as a root to obtain a rooted rank decomposition in the
convention of this paper.
\end{proof}

\begin{lemma}[CFI block-layout bound]
\label{lem:cfi-cutwidth-bound}
Let $B$ be a connected subcubic graph on at least two vertices, let $\pi$
be a layout of $B$, and let $U\subseteq V(B)$. Then
\[
  \lrw\bigl(\operatorname{CFI}(B,U)\bigr)
  \leq2\operatorname{cutw}_{\pi}(B).
\]
\end{lemma}

\begin{proof}
Order the fibres consecutively according to $\pi$, with an arbitrary order
inside each fibre. Across a cut between two fibres, every crossing base edge
contributes one adjacency block. Such a block has rank at most two because
adjacency depends only on equality of the two membership bits for that edge.
The cut rank is therefore at most
$2\operatorname{cutw}_{\pi}(B)$.

Consider a cut inside $M_U(v)$. Let $a$ and $b$ be the numbers of neighbours
of $v$ before and after $v$ in $\pi$. Let $c_0$ be the number of other base
edges crossing the cut immediately before $v$. Finally, let $s$ and $t$ be
the numbers of vertices of $M_U(v)$ on the two sides of the current cut.
The crossing matrix has three types of nonzero blocks. The other crossing
base edges contribute rank at most $2c_0$. The earlier-neighbour fibres to
the unplaced part of $M_U(v)$ contribute rank at most
$\min\{2a,t\}$. The placed part of $M_U(v)$ to the later-neighbour fibres
contributes rank at most $\min\{s,2b\}$. There is no block between the two
parts of $M_U(v)$ because this fibre is independent. Hence
\begin{equation}
  \rho\leq
  2c_0+\min\{2a,t\}+\min\{s,2b\}.
  \label{eq:cfi-block-layout-rank}
\end{equation}

If $v$ has degree one, its fibre has one vertex and there is no cut inside
it. If $v$ has degree two, then $s=t=1$ at the only internal cut and the
last two terms of \eqref{eq:cfi-block-layout-rank} sum to at most
$2\max\{a,b\}$. If $v$ has degree three, then $s+t=4$, so those terms sum
to at most $4\leq2\max\{a,b\}$. The base cuts immediately before and after
$v$ have sizes $c_0+a$ and $c_0+b$. Consequently,
\[
  \rho
  \leq2\bigl(c_0+\max\{a,b\}\bigr)
  =2\max\{c_0+a,c_0+b\}
  \leq2\operatorname{cutw}_{\pi}(B).
\]
This proves the claimed layout bound.
\end{proof}

The coefficient two cannot be improved for the class of layouts that keep
every fibre consecutive. If the crossing base edges at a cut form a
matching and all their endpoints have degree at least two, the CFI cut
matrix is a direct sum of equality blocks. Each relevant membership bit
assumes both values in its parity fibre, so every block has rank two. The
split-grid layout constructed in the proof of
Corollary~\ref{cor:linear-rank-width-asymptotic-range} has a cut after a
complete row with $n$ vertical matching edges and therefore cut rank $2n$.
No smaller leading coefficient can follow from a uniform per-edge estimate
for this block layout. A stronger bound would need additional global
structure, an interleaved layout, or a different construction.

\begin{proposition}[Rook and Shrikhande calibration]
\label{prop:rook-shrikhande-calibration}
For every $r\geq4$ and $s\geq5$,
\[
  D_{\mathrm{rw}}(r)\geq3,
  \qquad
  D_{\mathrm{lrw}}(s)\geq3.
\]
\end{proposition}

\begin{proof}
Let
\[
  R=\operatorname{CFI}(K_4,\varnothing),
  \qquad
  S=\operatorname{CFI}(K_4,\{v\}).
\]
Direct enumeration of the four parity words in each fibre identifies $R$
with $L(K_{4,4})$ and $S$ with
\[
  \operatorname{Cay}\bigl(\mathbb Z_4^2,
  \{\mathord{\pm}(1,0),\mathord{\pm}(0,1),
  \mathord{\pm}(1,1)\}\bigr).
\]
These are the $4$ by $4$ rook graph and the Shrikhande graph. Both are
strongly regular with parameters $(16,6,2,2)$. The diagonal, adjacent,
and nonadjacent ordered pairs therefore form the same stable
two-dimensional Weisfeiler-Leman colouring in both graphs. Yet the graphs
are nonisomorphic. Indeed, the neighbourhood of a vertex in $R$ is
$2K_3$, while the neighbourhood of a vertex in $S$ is $C_6$. Hence each
graph has Weisfeiler-Leman dimension at least three.

Every nonempty proper vertex cut of $K_4$ has all four base vertices as
endpoints of crossing edges. Hence $\ebw(K_4)=4$, and
Lemma~\ref{lem:cfi-endpoint-boundary-bound} gives
\[
  \rw(R),\rw(S)\leq4.
\]

Label the base vertices $v_1,v_2,v_3,v_4$, taking $v_1=v$ for $S$, and
order the fibres by this list. In the fibre at $v_i$, order the three
incident-edge bits by increasing other endpoint. With this coordinate
order, list an even fibre as
\[
  000,110,101,011
\]
and an odd fibre as
\[
  100,010,001,111.
\]
Put the four fibres consecutively, with the unique odd fibre first for $S$.
For a base edge $v_iv_j$, the block entry between parity words $s$ and $t$
is
\[
  1+s_{ij}+t_{ij}
\]
over $\mathbb F_2$. Thus each adjacency block is explicitly
$J+\sigma_{ij}\mathbf 1^{\mathsf T}
+\mathbf 1\tau_{ij}^{\mathsf T}$, restricted to the relevant parity
words. Gaussian elimination on these displayed parity-word lists gives the
following prefix ranks. Column $h$ means that $h$ vertices of the indicated
fibre have been placed, so the fourth column is the complete-fibre boundary.
\[
  \begin{array}{c|cccc}
    &h=1&h=2&h=3&h=4\\
    \hline
    v_1&1&2&3&3\\
    v_2&4&5&5&4\\
    v_3&5&4&4&3\\
    v_4&3&2&1&0
  \end{array}
\]
The same table holds for $R$ and $S$. It is an explicit width-five layout,
so $\lrw(R),\lrw(S)\leq5$. Monotonicity of the two extremal functions
proves the assertion. The bundled deterministic script
\texttt{verify\_calculations.py} reproduces the finite identifications, row
reductions, and prefix-rank table as a supplementary certificate. The proof
itself relies only on the explicit descriptions and certificates above.
\end{proof}

\begin{corollary}[Linear-rank-width asymptotic range]
\label{cor:linear-rank-width-asymptotic-range}
For every integer $r\geq2$,
\[
  \begin{cases}
    2,&2\leq r\leq4,\\
    \max\left\{3,\left\lfloor\dfrac r2\right\rfloor-1\right\},&r\geq5
  \end{cases}
  \leq D_{\mathrm{lrw}}(r)\leq2r+2.
\]
Consequently, $D_{\mathrm{lrw}}(r)=\Theta(r)$.
\end{corollary}

\begin{proof}
The upper bound is Corollary~\ref{cor:linear-rank-width-bound}. The cycle
$C_6$ used above has linear rank width at most two and Weisfeiler-Leman
dimension at least two. It supplies the constant lower term for every
$r\geq2$. Proposition~\ref{prop:rook-shrikhande-calibration} supplies the
lower term three for $r\geq5$.

Recall that a tree decomposition consists of vertex bags indexed by a tree.
Their union is the vertex set, every edge has both endpoints in one bag, and
the bags containing a fixed vertex induce a connected subtree. Its width is
the largest bag size minus one, and $\operatorname{tw}(G)$ is the minimum
such width.

For $n\geq2$, start with the $n$ by $n$ grid $Q_n$. Replace every interior
vertex $v$ by adjacent vertices $v^-$ and $v^+$. Attach the edges leading
up and left to $v^-$, and attach the edges leading down and right to $v^+$.
Leave the boundary vertices unchanged. Denote the resulting split grid by
$B_n$. This is the subcubic split grid used by Kiefer and Neuen
\cite[Section~6.2]{KieferNeuen2022}. It is connected, and contracting all
edges $v^-v^+$ recovers $Q_n$. The standard grid equality
$\operatorname{tw}(Q_n)=n$ \cite[Section~13.2]{Bodlaender1998} and minor
monotonicity give $\operatorname{tw}(B_n)\geq n$.

Order the rows of $Q_n$ from top to bottom and each row from left to right.
After $s$ vertices of a row, there are $n-s$ crossing vertical edges from
the preceding row, $s$ crossing vertical edges to the next row, and at most
one horizontal crossing edge. Thus this order has cutwidth at most $n+1$.
Expand each split interior vertex as the consecutive pair $v^-,v^+$. At
cuts before and after the pair, the crossing-edge count is unchanged. At
the cut inside the pair, the two incoming grid edges have closed and only
the edge $v^-v^+$ has opened, so the count decreases by one. Therefore this
layout of $B_n$ has cutwidth at most $n+1$.

The same order has at most $2n$ endpoints of crossing edges at every cut.
In an interior row, the $n$ vertical crossing edges have $2n$ endpoints,
and the endpoints of a horizontal crossing edge are already among them.
In the first or last row there are fewer vertical crossing edges, and a
horizontal edge adds at most one new endpoint. At the cut inside a split
pair, the two closed incoming edges are replaced by the internal edge, so
the number of crossing endpoints does not increase. A caterpillar on this
order also displays singleton cuts, each with at most four crossing endpoints.
Consequently,
\[
  \ebw(B_n)\leq2n.
\]
Lemma~\ref{lem:cfi-endpoint-boundary-bound} gives
\begin{equation}
  \rw\bigl(\operatorname{CFI}(B_n,U)\bigr)\leq2n
  \label{eq:split-grid-cfi-rank-width}
\end{equation}
for every $U\subseteq V(B_n)$.

Lemma~\ref{lem:cfi-cutwidth-bound} gives
\begin{equation}
  \lrw\bigl(\operatorname{CFI}(B_n,U)\bigr)\leq2n+2
  \label{eq:split-grid-cfi-linear-rank-width}
\end{equation}
for every $U\subseteq V(B_n)$.

Fix $x\in V(B_n)$. Neuen's nonisomorphism lemma shows that
$\operatorname{CFI}(B_n,\varnothing)$ and
$\operatorname{CFI}(B_n,\{x\})$ are nonisomorphic
\cite[Lemma~5]{Neuen2024Homomorphism}. Since
$\operatorname{tw}(B_n)\geq n$, his pebble-game lemma and the standard
homomorphism characterization of Weisfeiler-Leman show that the pair is not
distinguished by $(n-1)$-dimensional Weisfeiler-Leman
\cite[Lemma~12 and Theorems~10 and~14]{Neuen2024Homomorphism}, where
Theorem~10 is applied with parameter $n-1$. Thus each graph in the pair has
Weisfeiler-Leman dimension at least $n$.

For $r\geq6$, take $n=\lfloor r/2\rfloor-1$. Then $n\geq2$ and
$2n+2\leq r$. Equation~\eqref{eq:split-grid-cfi-linear-rank-width} gives an
admissible graph of dimension at least
$\lfloor r/2\rfloor-1$. Together with the $C_6$ baseline, this proves the
lower bound for every $r\geq2$.
\end{proof}

\begin{corollary}[Rank-width asymptotic range]
\label{cor:rank-width-asymptotic-range}
For every integer $r\geq2$,
\[
  \begin{cases}
    2,&2\leq r\leq3,\\
    \max\left\{3,\left\lfloor\dfrac r2\right\rfloor\right\},&r\geq4
  \end{cases}
  \leq D_{\mathrm{rw}}(r)\leq3r+1.
\]
Consequently, $D_{\mathrm{rw}}(r)=\Theta(r)$.
\end{corollary}

\begin{proof}
The upper bound is Corollary~\ref{cor:rank-width-bound}. The cycle $C_6$
gives the lower term two. For $r\geq4$, put $n=\lfloor r/2\rfloor$.
The CFI pair used in the proof of
Corollary~\ref{cor:linear-rank-width-asymptotic-range} has
Weisfeiler-Leman dimension at least $n$, while
\eqref{eq:split-grid-cfi-rank-width} gives rank width at most
$2n\leq r$. Proposition~\ref{prop:rook-shrikhande-calibration} supplies
the lower term three.
\end{proof}

\begin{corollary}[Algorithmic consequences]
\label{cor:algorithmic-consequences}
For every fixed integer $k\geq1$, isomorphism of $n$-vertex graphs of rank
width at most $k$ can be decided in time
\[
  O\bigl(n^{3k+2}\log n\bigr),
\]
and the class admits canonization in time
\[
  O\bigl(n^{3k+4}\log n\bigr).
\]
For graphs of linear rank width at most $k$, the corresponding bounds are
\[
  O(n^{2k+3}\log n)
  \qquad\text{and}\qquad
  O(n^{2k+5}\log n).
\]
\end{corollary}

\begin{proof}
The standard implementation of $d$-dimensional Weisfeiler-Leman runs in
$O(n^{d+1}\log n)$ time \cite{ImmermanLander1990}. Apply
Corollaries~\ref{cor:rank-width-bound} and
\ref{cor:linear-rank-width-bound} with $d=3k+1$ and $d=2k+2$,
respectively. The proof applies to arbitrary finite initial vertex
colourings by Remark~\ref{rem:uniform-initial-colours}, and
the dimension bound is independent of the number of initial colour classes.
The generic canonization construction
\cite[Theorem~2.7]{GroheNeuen2023} runs in time
$O(n^{d+3}\log n)$, which gives the other two bounds.
\end{proof}

\section{Conclusion}
\label{sec:conclusion}

Theorem~\ref{thm:main} gives dimensions $3k+1$ for rank width at most $k$
and $2k+2$ for linear rank width at most $k$. It improves the unrestricted
$3k+3$ estimate of Levet, Rombach, and Sieger by two dimensions
\cite[Corollary~3.12]{LevetRombachSieger2024Full}, and the published $3k+4$
bound of Grohe and Neuen by three dimensions \cite{GroheNeuen2023}. The proof
retains the three incident cut ranks, synchronizes the child transition with
two remote keys, and requires no decomposition of the target graph.

The optimized profile records information not captured by rank width or
linear rank width alone. Its coordinates obey exact maximum identities and
remain fixed while linear rank width grows arbitrarily. At rank width two the
only pairs are $(5,6)$ and $(6,7)$. The parity-fibre CFI construction proves
$D_{\mathrm{lrw}}(r)=\Theta(r)$ and gives a linear lower bound for
$D_{\mathrm{rw}}(r)$, with the rook and Shrikhande graphs providing the
small-parameter calibration. This links the width-bounded setting to broader
CFI and homomorphism methods
\cite{Roberson2026,GroheLichterNeuenSchweitzer2025}.

The result concerns dimension under unrestricted stabilization. It does not
bound quantifier depth or produce fixed-point canonization. The bounded-round
result of Levet, Rombach, and Sieger is complementary
\cite{LevetRombachSieger2024}. The exact value of $D_2$ remains open. Within
the present method, the obstruction to coefficient two is a fork with three
large incident cut ranks. Further progress may require encoding one interface
in a synchronized palette before releasing its physical anchors, narrowing
the constant-factor gaps in the CFI bounds, or combining low dimension with
bounded refinement depth.

\section*{Reproducibility materials}

\ifdefined\WLAncillaryMaterials
The ancillary files contain \texttt{verify\_calculations.py} and instructions
for running it. The script verifies the finite enumerations and calibration
examples used in this paper, but does not machine-check the general proofs.
\else
The supplementary archive contains source files, build instructions, and
scripts that verify the finite enumerations and calibration examples used in
this paper. The scripts do not machine-check the general proofs.
\fi

\bibliographystyle{plainnat}
\bibliography{mybib}

\end{document}